\documentclass[10pt,a4paper,oneside]{article}
\usepackage[english]{babel}
\usepackage{a4wide,amsmath,amsthm,amssymb,url,graphicx,xspace,algorithm,algorithmic,pgf}
\usepackage[normalem]{ulem}
\usepackage[latin1]{inputenc}
\usepackage{enumitem}
\usepackage{multirow}
\usepackage{hyperref}
\usepackage{tabularx}
\usepackage{tikz}
\usepackage{cases}
\usepackage{bm}

\def\F{\mathbb{F}}
\def\C{\mathrm{C}}

\def\L{\mathcal{L}}
\def\K{\mathcal{K}}
\def\A{\mathcal{A}}
\def\S{\mathcal{S}}
\def\P{\mathcal{P}}
\def\B{\mathcal{B}}
\def\L{\mathcal{L}}
\def\V{\mathrm{V}}

\DeclareMathOperator{\supp}{supp}

\DeclareMathOperator{\PG}{PG}

\DeclareMathOperator{\PGL}{PGL}

\renewcommand{\c}{{\bf c}}
\renewcommand{\v}{{\bf v}}

\newtheorem{theorem}{Theorem}[section]

\newtheorem{lemma}[theorem]{Lemma}
\newtheorem{result}[theorem]{Result}
\newtheorem{proposition}[theorem]{Proposition}

\newtheorem{corollary}[theorem]{Corollary}
\newtheorem{notation}[theorem]{Notation}
\newtheorem{example}[theorem]{Example}

\theoremstyle{definition}
\newtheorem{definition}[theorem]{Definition}
\newtheorem{remark}[theorem]{Remark}

\newtheoremstyle{dotless}{}{}{\itshape}{}{\bfseries}{}{ }{}

\theoremstyle{dotless}
\newtheorem*{mtheorem}{Main Theorem}

\newcommand{\comments}[1]{}

\author{Maarten De Boeck\thanks{This author is supported by the Croatian Science Foundation under the project 5713.} \and Geertrui Van de Voorde\thanks{This author is supported by the Marsden Fund Council administered by the Royal Society of New Zealand (MFP-UOC1805).}}
\title{Embedded antipodal planes and the minimum weight of the dual code of points and lines in projective planes of order $p^2$}

\date{}
\begin{document}
\maketitle
\begin{abstract} The minimum weight of the code generated by the incidence matrix of points versus lines in a projective plane has been known for over 50 years. Surprisingly, finding the minimum weight of the dual code of projective planes of non-prime order is still an open problem, even in the Desarguesian case. 

In this paper, we focus on the case of projective planes of order $p^2$, where $p$ is prime, and we link the existence of small weight code words in the dual code to the existence of embedded subplanes and {\em antipodal planes}. In the Desarguesian case, we can exclude such code words by showing a more general result that no antipodal plane of order at least $3$ can be embedded in a Desarguesian projective plane.

Furthermore, we use combinatorial arguments to rule out the existence of code words in the dual code of points and lines of an arbitrary projective plane of order $p^2$, $p$ prime, of weight at most $2p^2-2p+4$ using more than two symbols. In particular, this leads to the result that the dual code of the Desarguesian projective plane $\PG(2,p^2)$, $p\geq 5$, has minimum weight at least $2p^2-2p+5$.
\end{abstract}

\noindent
{\bf Keywords: projective plane, antipodal plane,(dual) code of projective plane, minimum weight}\\
{\bf MSC2020: 51E22, 51A45,94B05}

\section{Introduction} 

Let $\V(n,p)$ denote the $n$-dimensional vector space over the finite field $\F_p$, where $p$ is prime. A {\em $p$-ary linear code $\C$} is a $k$-dimensional subspace $\C$ of $\V(n,p)$. The code $\C$ is said to have {\em distance} $d$ if $d$ is the smallest (Hamming) distance between two vectors (called \emph{code words}) of $\C$. For any code $\C$, the orthogonal subspace (with respect to the standard dot product of $\V(n,p)$) is also a $p$-ary linear code, called the \emph{dual code} and denoted by $\C^\bot$. The study of the parameters of the code of projective planes dates back to the 1960's.

Throughout this paper, we denote the Desarguesian projective plane of order $q$, $q=p^h$, $p$ prime, by $\PG(2,q)$. 
Let $\Pi$ be a (not necessarily Desarguesian) projective plane of order $q$, $q=p^h$, where $p$ is prime. An incidence matrix $A = (a_{ij})$ of points and lines of $\Pi$ is a matrix whose rows and columns are indexed by the lines and points of $\Pi$, respectively, and with entry
\[
	a_{ij} =
	\begin{cases}
		1 & \text{if point $j$ belongs to line $i$,}\\
		0 & \text{otherwise.}
	\end{cases} 
\]
The {\em $p$-ary code} of $\Pi$ is the $\mathbb{F}_p$-span of the rows of the chosen incidence matrix $A$ with a fixed ordering of the points; the code is a subspace of $\V(q^2+q+1,p)$. Throughout this paper, we always consider the $p$-ary code of $\Pi$, for the prime $p$ such that $q=p^h$, and we denote this code by $\C(\Pi)$, or in the case that $\Pi=\PG(2,q)$, by $\C(2,q)$. The {\em support} of a code word $\c$, denoted by $supp(\c)$, is the set of all non-zero positions of the vector $\c$. Since every position corresponds to a point of $\Pi$, we consider $supp(\c)$ as a set of points. The {\em weight} of $\c$, denoted by $w(\c)$, is the size of $supp(\c)$. It is easy to see that the minimum (non-zero) weight and the minimum (non-zero) distance coincide since $\C(\Pi)$ is linear. We let $(\c_1,\c_2)$ denote the dot product in $\mathbb{F}_p$ of two code words $\c_1, \c _2$ of $\C(\Pi)$. Furthermore, we identify a set of points $\S$ of $\Pi$ with its incidence vector. The dual code of $\C(\Pi)$, denoted by $\C(\Pi)^\bot$ (or $\C(2,q)^\bot$ if $\Pi=\PG(2,q)$), is the set of all vectors orthogonal to all code words of $\C(\Pi)$. So,
$$\C(\Pi)^\bot=\{ \v\in \V(q^2+q+1,p) \mid (\v,\c)= 0,\ \forall \c\in \C(\Pi)\}.$$
It is easy to see that $\c\in\C(\Pi)^{\bot}$ if and only if $(\c,\ell) = 0$ for all lines $\ell$ of $\Pi$, using the conventions introduced above.

For most projective planes, finding the dimension of $\C(\Pi)$ is still an open problem. It has been proven that the dimension of $\C(2,p^{h})$ is equal to $\binom{p+1}{2}^{h}+1$ (see e.g. \cite[Theorem 5.7.1]{AK}), and the Hamada-Sachar conjecture states that for any projective plane $\Pi$ of order $p^{h}$, the code $\C(\Pi)$ has dimension at least $\binom{p+1}{2}^{h}+1$. This conjecture has been verified for $h=1$ (see \cite{sachar}) and for some small values of $q$ in case of translation planes (see \cite{AK} and \cite[Section 3]{keymac}).

The minimum weight of $\C(\Pi)$, as well as the nature of the minimum weight code words, can be easily deduced in a geometric way:
\begin{result}(see e.g. \cite[Theorem 6.3.1]{AK}) 
	Let $\Pi$ be a projective plane of order $n$. The minimum weight of $\C(\Pi)$ is $n+1$ and the minimum weight code words are the scalar multiples of the incidence vectors of the lines of $\Pi$. 
\end{result}
The situation for $\C(\Pi)^\bot$ is more complicated: its minimum weight remains unknown in general, even for the Desarguesian case. We have the following result for the Desarguesian case; considering the difference of two lines shows that the upper bound is valid for any projective plane.

\begin{result} (see e.g. \cite[Theorem 6.4.2]{AK})\label{res:ak} The minimum weight $d^\perp$ of $\C(2,q)^\perp$ satisfies
$$q+p\leq d^\perp\leq 2q,$$ with equality in the lower bound for $q$ even.
\end{result}

It follows that the minimum weight of $\C(2,p)^\bot$, $p$ prime, is $2p$, and the minimum weight of $\C(2,2^h)^\bot$ is $2^h+2$. Apart from these two cases, only the cases $q=9$ \cite{p3} (where $d^\perp=15$) and $q=25$ \cite[Corollary 1.2]{p5} (where $d^\perp=45$) have been fully settled. For $q=49$, it is only known that $88\leq d^\perp\leq 91$ \cite{p7}.

For arbitrary projective planes, we have the following result.
\begin{result}\label{bag}\cite[Theorem 2, adapted here for $\C(\Pi)$]{bagchi} Let $\Pi$ be a projective plane of order $q=p^h$, $p$ prime. The minimum weight $d^\perp$ of $\C(\Pi)^\perp$ is at least
$$2\left(q+1-\frac{q}{p}\right)$$
and in case of equality, the incidence system induced by $\Pi$ on the support of a code word $\c$ of minimum weight is of the form $D_1*D_2$ where $D_1$ and $D_2$ are disjoint $2-(q+1-\frac{q}{p},p,1)$-designs embedded in $\Pi$. (A $2-(q+1-\frac{q}{p},p,1)$-design is a point-line incidence structure with $q+1-\frac{q}{p}$ points such that any two distinct points are incident with exactly one line and every line is incident with $p$ points. Furthermore, $D_1*D_2$ denotes the {\em join} of $D_1$ and $D_2$, which has as point set $D_1\cup D_2$, and as line set (i) the lines of $D_1$ and $D_2$ and (ii) the lines $\{P,Q\}$ where $P\in D_1$ and $Q\in D_2$). Furthermore, only two symbols, $\lambda$ and $-\lambda$, are used in $\c$ and the point sets of $D_i$, $i=1,2$ each consist of the set of positions on which $\c$ has a fixed symbol, either $\lambda$ or $-\lambda$.
\end{result}
In this paper, we will focus on the case $q=p^2$, $p$ prime. Recall from Result \ref{res:ak} that the minimum weight of $C(2,2^2)$ is $6$, so we consider the case $p\geq 3$. Result \ref{bag} shows that when $\Pi$ has order $p^2$, the minimum weight $d^\perp$ of $\C(\Pi)^\perp$ is at least $2p^2-2p+2$ and in case of equality, the support of a code word of minimum weight defines two embedded $2-(p^2-p+1,p,1)$-designs.

A $2-(p^2-p+1,p,1)$-design with $p\geq 3$ is precisely a projective plane of order $p-1$. In the case that $\Pi=\PG(2,p^2)$, we know that the only proper subplanes of $\PG(2,p^2)$ have order $p$. It follows that the lower bound in Result \ref{bag} is not sharp in this case:

\begin{corollary}\label{cordesmin}
	The minimum weight $d^\perp$ of $\C(2,p^2)^\perp$, $p\geq 3$, is at least $2p^2-2p+3$.
\end{corollary}

In this paper, we will show the following.

\begin{mtheorem}\label{mt:maintheorem}
	Let $\Pi$ be an arbitrary projective plane of order $p^{2}$, $p\geq 7$, $p$ prime, and let $d^\perp$ be the minimum weight of $\C(\Pi)^\perp$. Then $d^\perp=2p^2-2p+2$, or $d^\perp=2p^2-2p+4$ or $d\geq 2p^2-2p+5.$ Furthermore:
	 \begin{itemize}
	 \item 	 if $d^\perp=2p^2-2p+2$ then every minimal weight code word is of the form $\lambda(\v_1-\v_2)$ with $\lambda\in\F_{p}^{*}$ and $\v_1,\v_2$ the incidence vectors of two disjoint projective planes of order $p-1$ embedded in $\Pi$;
	 \item if $d^\perp=2p^2-2p+4$ then every minimal weight code word is of the form $\lambda(\v_1-\v_2)$ with $\lambda\in\F_{p}^{*}$ and $\v_1,\v_2$ the incidence vectors of two disjoint antipodal planes of order $p-1$ embedded in $\Pi$.
	 \end{itemize}
	  Finally, $d^\perp\leq 2p^2-p$ if $\Pi$ contains a Baer subplane.
\end{mtheorem}
Note that the first bullet point follows from Result \ref{bag}. In the Desarguesian case $\C(2,p^2)^\bot$, $p\geq 3$, we have already mentioned that we can exclude the possibility of embedded projective planes of order $p-1$. In this paper (see Theorem \ref{noembeddedantipodal}), we will also show that there are no antipodal planes of order $p-1$ embedded in $\PG(2,p^2)$. This leads to the following corollary of our main theorem.
\begin{corollary}\label{cor:main}
	The minimum weight $d^\perp$ of $\C(2,p^2)^\perp$, $p\geq7$, satisfies
	\begin{align}\label{main}2p^2-2p+5\leq d^\perp\leq 2p^2-p.\end{align}
\end{corollary}

The bound \eqref{main} is valid for $p=7$. As mentioned above, in \cite{p7}, the authors investigated the dual codes of arbitrary planes of order $49$. In particular, their result shows that in the Desarguesian case, $88\leq d^\perp\leq 91$; our lower bound equals $89$, so our result is a slight improvement of their result for $\C(2,7^2)^\bot$. Note that the minimum weight of $\C(2,5^2)^\bot$ is $45$, so the bound in Equation \eqref{main} is valid for $p=5$ too. This was shown in \cite{p5}, but not all arguments in our paper work for $p=5$.

\section{Antipodal planes}
\subsection{Preliminaries}

A {\em partial linear space} is an incidence structure $(\P,\L)$ with point set $\P$ and line set $\L$, where every line is a subset of $\P$ of size at least two, such that two distinct points are incident with at most one line. Antipodal planes are a class of partial linear spaces which are similar to projective planes. We will use the following definition of an antipodal plane.

\begin{definition}
	An {\em antipodal plane} of order $s\geq 2$ is a partial linear space $(\P,\L)$ such that
	\begin{itemize}
\item $|\P|=|\L|=s^2+s+2$;
\item every line contains $s+1$ points;
\item every point lies on $s+1$ lines.
\end{itemize}
\end{definition}

It easily follows that for each point $P\in \P$, there is a unique point, called the {\em antipodal point} and denoted by $P^\perp$, which does not lie on a line of $\L$ through $P$. Note that $P^{\perp\perp}=P$. Likewise, for every line $\ell$ of $\L$ there is a unique line,  called the {\em antipodal line}, which does not meet $\ell$ in a point. We denote this line by $\ell^\perp$. As for points, $\ell^{\perp\perp}=\ell$.

\begin{lemma}\label{perpcoll} Let $\Pi$ be an antipodal plane $(\P,\L)$ of order $s$. Let $P\in \P$ and $\ell\in\L$. If $P\in \ell$, then $P^\perp\in \ell^\perp$, so $\ell^\perp$ consists of the $s+1$ antipodal points of the points of $\ell$.
\end{lemma}
\begin{proof} There are $s+1$ points on $\ell^\perp$. Since $P\in \ell$ and $\ell$ does not have a point in common with $\ell^\perp$, there are at most $s$ lines through $P$ which intersect $\ell^\perp$. That means that there is at least one point on $\ell^\perp$ not collinear to $P$. Since the unique point not collinear to $P$ is $P^\perp$, we conclude that $P^\perp\in \ell^\perp$. Since for every point of $\ell$, its antipodal point lies on $\ell^\perp$, we conclude that $\ell^\perp$ consists precisely of the antipodal points of the points on $\ell$.
\end{proof}

Different terminology and slightly different definitions occur in the literature, see e.g. Schleiermacher \cite[Theorem 2.2,2.3]{schleier} for equivalent definitions and the notion of {\em regularity} of antipodal planes (which follows from finiteness and the definitions in that paper). Dembowski \cite[Section 7.4]{dembowski} described antipodal planes as {\em elliptic semi-planes} with parallel classes of size $2$, and linked them to biplanes (i.e. certain $2$-designs with $\lambda=2$). Moreover, antipodal planes are a particular type of what Payne \cite{payne} called {\em nearly generalised $3$-gons}.

It is not known for which orders there exists an antipodal plane. For $s=2,3$ there are unique examples which we will describe below. It is conjectured in \cite[Section 7.4 (13)]{dembowski} and \cite{payne} that no antipodal plane of order different from $2$ or $3$ exists. In \cite{payne}, the author derives some conditions on the order of an antipodal plane, but there are still infinitely many orders satisfying those conditions. In this section, we will show that no antipodal plane of order $s>3$ can be embedded in a Desarguesian projective plane.

\begin{example}\label{exorder2} Let $q=3^h$ or $3|q-1$, let $\omega$ be an element in $\F_q$ with $\omega^2-\omega+1=0$ -- which exists for all such $q$ -- and consider the projective plane $\PG(2,q)$. The {\em M\"obius-Kantor configuration} consists of the following $8$ points: 
$(1,0,0)$, $(0,1,0)$, $(0,0,1)$, $(1,1,0)$, $(0,1,\omega)$, $(1,1,1)$, $(\omega,1,1)$, and $(1,0,1-\omega)$.
It is easy to check that there are precisely $8$ lines of $\PG(2,q)$ containing $3$ points of these 8 points, and that these 8 points and 8 lines form an antipodal plane of order $2$, embedded in $\PG(2,q)$. We will show in Proposition \ref{case2} that it is impossible to embed this antipodal plane in a Desarguesian projective plane of an order not satisfying the conditions on $q$ expressed at the beginning of this Example.
\end{example}

\begin{example}\label{nieuw} (see also \cite[Section 7.4 (13)]{dembowski}) Consider $\Pi=\PG(2,4)$ and a subplane $\pi$ of order 2 in $\Pi$. Let $\P$ be the set of points of $\Pi\setminus \pi$ and let $\L$ be the set of lines of $\Pi$ that are not extended lines of the subplane $\pi$. Then $|\P|=|\L|=14$, every point of $\S$ lies on 4 lines and every line of $\L$ contains 4 points, so $(\P,\L)$ is an antipodal plane of order $3$. Since we can embed $\PG(2,4)$ in $\PG(2,2^h)$ for every even $h$, we see that we can embed an antipodal plane of order $3$ in a Desarguesian projective plane of order $2^h$, $h$ even. We will show in Proposition \ref{case3} that it is impossible to embed this antipodal plane in a Desarguesian projective plane of a different order.

\end{example}

\begin{remark}\label{opm}
	The examples of orders $2$ and $3$ are unique up to isomorphism (see \cite[Corollary 5.5]{schleier}). The example of order 2 can also be constructed as follows. Define an $8\times 8$-matrix where the first row has a 1 at positions 1, 2 and 4 and 0 elsewhere, and the other seven rows are obtained from applying a cyclic shift of length one to the previous row. This matrix defines an incidence matrix of an incidence structure $(\P,\L)$ with 8 points $\P$ corresponding to columns and 8 rows corresponding to lines $\L$. The order of the points in Example \ref{exorder2} is chosen to correspond to this incidence matrix.
	
	\par The example of order 3 can also be constructed as follows. Define a $14\times 14$-matrix where the first row has a $1$ at positions $1,2,5,7$ and 0 elsewhere, and the other $13$ rows are obtained from applying a cyclic shift of length one to the previous row. This matrix defines an incidence matrix of an incidence structure $(\P,\L)$ with $14$ points $\P$ corresponding to columns and $14$ rows corresponding to lines $\L$.  It is clear that every line contains $4$ points, and it can be checked that every point lies on $4$ lines;  since it is a partial linear space, we find that $(\P,\L)$ is an antipodal plane of order 3.
\end{remark}

\subsection{Antipodal planes embedded in Desarguesian projective planes}

Let $A_1,A_2,A_3$ be the points with coordinates $(1,0,0),(0,1,0),(0,0,1)$ in $\PG(2,q)$. The lines through these points are of the form $X_3=t_1X_2,X_1=t_2X_3,X_2=t_3X_1$ respectively, and we call $t_i$ the {\em slope} of the corresponding line and let $(A_iB)$ denote the slope of the line $A_iB$.
Recall the following classical theorems, attributed to Menelaos and Ceva respectively. The proofs are an easy exercise.
\begin{lemma} \label{menelaos}(Menelaos) Let $\ell$ be a line not through $A_1,A_2$ nor $A_3$. If $\ell$ meets $A_2A_3$ in $B_1$, $A_1A_3$ in $B_2$, $A_1A_2$ in $B_3$, then $$(A_1B_1)(A_2B_2)(A_3B_3)=-1.$$ 
\end{lemma}

\begin{lemma}\label{ceva} (Ceva) For every point $X$ not on the sides of the triangle $A_1A_2A_3$,  we have $$(A_1X)(A_2X)(A_3X)=1.$$ 
\end{lemma}

We say that the incidence structure $(\P,\L)$ is {\em embedded} in a projective plane $\Pi=(\P',\L')$ if there are injective maps $\phi:\P\mapsto \P'$ and $\psi:\L\mapsto \L'$ which preserve incidence and non-incidence. This means that we require that if there is no line in $\L$ containing $P$ and $Q$ of $\P$, then there is no line of $\psi(\L)$ containing $\phi(P)$ and $\phi(Q)$. Hence, if we embed an antipodal plane $(\P,\L)$ in a projective plane $(\P',\L')$, the line of $\L'$ containing the (images of) the points $P$ and $P^\perp$ is different from the lines obtained from the embedding of $(\P,\L)$.

In the following proofs, we omit the embedding maps $\phi$ and $\psi$ in the notation and simply consider the points of the embedded antipodal plane as points of the projective plane $\PG(2,q)$ in which it is embedded.

\begin{proposition}\label{case2}
	An antipodal plane $(\P,\L)$ of order $2$ is embeddable in a projective plane $\PG(2,q)$ if and only if $q=3^h$ or $3|q-1$.
\end{proposition}
\begin{proof}
	Suppose that there is an antipodal plane $(\P,\L)$ embedded in $\PG(2,q)$. Consider a line $\ell$ of $\L$ and let $P_1,P_2,P_4$ denote the points of $\P$ on $\ell$. Let $P_1^\perp=P_5,P_2^\perp=P_6,P_4^\perp=P_8$ be their antipodal points; these are contained in $\ell^\perp$ and hence, collinear. Let $P_3$ denote the third point of $\P$, different from $P_2$ and $P_5$ on $P_2P_5$ and let $P_7=P_3^\perp$. Then $P_6=P_2^\perp,P_7=P_3^\perp,P_1=P_5^\perp$ are collinear. This implies that $P_1P_3$ cannot meet $\ell^\perp$ in the point $P_6$, and since $P_5=P_1^\perp$, it follows that $P_1P_3$ meets $\ell^\perp$ in $P_8$. Since $P_1,P_3,P_8$ are collinear, $P_5=P_1^\perp,P_7=P_3^\perp, P_4=P_8^\perp$ are collinear. Furthermore, there is a line of $\L$ through $P_2$ and $P_8$, which then necessarily contains the point $P_7$; and finally, this implies that $P_6=P_2^\perp$, $P_3=P_7^\perp$ and $P_4=P_8^\perp$ are collinear. (With this ordering of the points, we see that the incidence matrix of points and lines of $(\P,\L)$ is indeed given by the cyclic $8\times 8$ $(0,1)$-matrix whose first row has non-zero entries in positions $1,2,4$ as in Remark \ref{opm}.)
	
	Now suppose that $(\P,\L)$ is embedded in $\PG(2,q)$. Since $\PGL(3,q)$ acts transitive on the frames of $\PG(2,q)$, we can take $P_1=(1,0,0)$, $P_2=(0,1,0)$, $P_3=(0,0,1)$ and $P_6=(1,1,1)$. This implies that $P_4=(1,1,0)$ and $P_5=(0,1,\omega)$ for some $\omega\in \F_q$. Since $P_5,P_6,P_8$ are collinear, and $P_8$ lies on $P_1P_3$, we find $P_8=(1,0,1-\omega)$. The lines $P_4P_5$ and $P_1P_6$ meet in $P_7$, which implies that $P_7=(\omega-1,\omega,\omega)$. Note that the choice of these coordinates implies that seven of the triples of collinear points in $\P$ are indeed collinear in $\PG(2,q)$. Finally, the points of the last line, that is $P_2,P_7,P_8$, are collinear if and only if $\omega^2-\omega+1=0$. We see that $\omega$ exists if and only if $q=3^h$ (in which case $\omega=-1$) or $3|q-1$.
\end{proof}

\begin{lemma}\label{driehoek}
	Suppose that $(\P,\L)$ is an antipodal plane of order $s\geq 3$ embedded in $\PG(2,q)$, then there are $3$ points, $A_1,A_2,A_3\in \P$ forming a triangle whose sides are lines of $\L$ and such that $A_1^\perp,A_2^\perp,A_3^\perp$ are not  on the sides of the triangle.
\end{lemma}
\begin{proof}
	Let $\ell\in \L$ be a line and let $A_1,A_2$ be points of $\P$ on $\ell$. The points $A_1^\perp$ and $A_2^\perp$ do not lie on $\ell$. Let $m$ be a line of $\L$ through $A_1$, different from $\ell$ and not through $A_2^\perp$ (which exists since $s+1> 2$). Let $A_3$ be a point of $m$, different from $A_1$, different from the intersection point of $\ell^\perp$ with the line $m$ and different from the intersection point of the line through $A_1^\perp$ and $A_2$ with $m$. The point $A_3$ exists since $s+1\geq 4$. We have chosen $A_3$ such that $A_2A_3$ does not contain $A_1^\perp$. Since $A_3$ does not lie on $\ell^\perp$, Lemma \ref{perpcoll} implies that $A_3^\perp$ does not lie on $\ell=A_1A_2$. Furthermore, the line $m=A_1A_3$ does not contain $A_2^\perp$ by construction. We conclude that the sides of the triangle $A_1,A_2,A_3$ are lines of $\L$ and that the points $A_1^\perp,A_2^\perp, A_3^\perp$ are not on the sides of this triangle.
\end{proof}

\begin{proposition}\label{case3}
	An antipodal plane $(\P,\L)$ of order $3$ is embeddable in a projective plane $\PG(2,q)$ if and only if $q=2^h$ with $h$ even (as in Example \ref{nieuw}).
\end{proposition}
\begin{proof}
	Suppose that there is an antipodal plane $(\P,\L)$ of order $3$ embedded in $\PG(2,q)$. By Lemma \ref{driehoek}, we can find $A_1,A_2,A_3\in \P$ forming a triangle whose sides are lines of $\L$ and such that $A_1^\perp,A_2^\perp,A_3^\perp$ are not contained on the sides of the triangle. Since $\PGL(3,q)$ acts transitively on the frames of $\PG(2,q)$, we can pick coordinates such that $A_1=(1,0,0)$, $A_2=(0,1,0)$, $A_3=(0,0,1)$ and $A_1^\perp=(1,1,1)$. Let $\ell_1=A_2A_3$, $\ell_2=A_1A_3$, $\ell_3=A_1A_2$ and let $\ell_3^\perp\cap \ell_2=X_2$, $\ell_3^\perp\cap \ell_1=X_1$, $\ell_2^\perp\cap \ell_3=X_3$, $\ell_2^\perp\cap \ell_1=Y_1$, $\ell_1^\perp\cap \ell_3=Y_3$, $\ell_1^\perp\cap \ell_2=Y_2$. It is easy to check that the points $X_i,Y_i$ are all distinct and form  together with $A_1,A_2,A_3,A_1^\perp,A_2^\perp,A_3^\perp$ the set of all 12 points of $\P$ contained on the lines $\ell_1,\ell_2,\ell_3,\ell_1^\perp,\ell_2^\perp,\ell_3^\perp$.

	The line $\ell_2$ contains precisely $4$ points of $\P$, namely $A_1$, $A_3$, $X_2$, and $Y_2$. Since $A_2\notin \ell_2^\perp$, the point $A_2$ is collinear to all four of them. It is collinear to $A_1$ via $\ell_3$ and to $A_3$ via $\ell_1$. Furthermore, $A_1^\perp$ is collinear to $X_2$ via $\ell_3^\perp$, which meets $\ell_1$ in $X\neq A_2$, so $A_2A_1^\perp$ cannot meet $\ell_2$ in $X_2$. It follows that the line $A_2A_1^\perp$ meets $\ell_2$ in $Y_2$. Hence, we have $Y_2=(1,0,1)$ and since $A_3A_1^\perp$ meets $\ell_3$ in $Y_3$,  we have $Y_3=(1,1,0)$. Now let $X_3=(1,t,0)$ with $t\neq 0,1$. Since $X_3A_1^\perp$ meets $\ell_1$ in $Y_1$,  we have $Y_1=(0,1-t,1)$. Since $A_3^\perp$ lies on $Y_2Y_3$ and $X_3Y_1$, we find that $A_3^\perp=(2-t,1,1-t)$. Since $A_2^\perp$ lies on $A_1Y_1$ and $Y_2Y_3$, $A_2^\perp=(2-t,1-t,1)$. We find $X_{2}=(1,0,t)$ since $X_{2}=\ell_3^\perp\cap \ell_2$ and $\ell_3^\perp$ is the line through $A_1^\perp$ and $A_2^\perp$. Expressing that $A_2,A_3^\perp$ and $X_2$ are collinear leads to 
\begin{align}\label{t3s}
	t^2-3t+1=0.
\end{align}
Now let $Z_1$ be the point of $\P$ on the line of $\L$ through $A_1A_2^\perp$, different from $A_1,A_2^\perp,Y_1$ and let $Z_2$ be the point of $\P$ on the line of $\L$ through $A_1A_3^\perp$, different from $A_1,A_3^\perp,X_1$. Then one of $Z_1,Z_2$ lies on the line $A_3Y_3$ and the other on the line $A_3X_3$; similarly, one of $Z_1,Z_2$ lies on the line $A_2Y_2$ and the other on the line $A_2X_2$.

If $Z_1$ lies on $A_3X_3$, then $Z_1=(1-t,t(1-t),t)$ and $Z_2=(1,1,1-t)$. The line $A_2Y_2$ has equation $X=Z$ and needs to contain one of $Z_1,Z_2$. 
If $Z_2$ lies on $A_2Y_2$, then $t=0$, a contradiction. This implies that $Z_1$ lies on $A_2X_2$. It follows that $t=1-t$, an immediate contradiction if $\mathrm{char}( \F_q)=2$, and $t=\frac{1}{2}$ if $\mathrm{char}(\F_q)\neq 2$. But plugging in $t=\frac{1}{2}$ in Equation \eqref{t3s} leads to a contradiction.

We conclude that $Z_1$ lies on $A_3Y_3$. It follows that $Z_1=(1-t,1-t,1)$ and $Z_2=(1,t,(1-t)t)$. As before, the line $A_2Y_2$ has equation $X=Z$ and needs to contain one of $Z_1,Z_2$. 
If $Z_1$ lies on $A_2Y_2$, then $t=0$, a contradiction. This implies that $Z_1$ lies on $A_2X_2$. But $A_2X_2$ has equation $Z=tX$, so it follows that $t(1-t)=1$, and hence
\begin{align}\label{tb}
	t^2-t+1=0\:.
\end{align}
Considering Equation \eqref{t3s}, this is a contradiction if $\mathrm{char}(\F_q)\neq 2$. We conclude that $q=2^h$. Since there needs to exists a root $t$ to Equation \eqref{tb} in $\F_{2^h}$, $\F_{2^h}$ needs to contain the subfield $\F_4$, so $h$ is necessarily even. If $t^2+t+1=0$, all coordinates of the $14$ points defined in this proof belong to $\F_{4}$ and the $7$ points of the subplane $\PG(2,4)$ that do not belong to the antipodal plane have coordinates $(1,0,t+1),(1,1,t+1),(1,t,t),(1,t+1,0),(1,t+1,1),(1,t+1,t+1),(0,1,1)$. It is easy to check that these latter $7$ points form a Fano subplane. Hence, the $14$ points define the antipodal plane of order $3$ described in Example \ref{nieuw}.
\end{proof}

\begin{remark}
	One can also use the fact that the antipodal planes of orders $2$ and $3$ are unique (see \cite{schleier}) to derive Propositions \ref{case2} and \ref{case3}. However, the above proofs don't use this unicity result and  also slightly stronger: Proposition \ref{case2} reproofs the unicity result, and Proposition \ref{case3} shows that there is a unique embeddable antipodal plane of order $3$.
\end{remark}

\begin{theorem}\label{noembeddedantipodal}
An antipodal plane $(\P,\L)$ of order $s\geq 4$ cannot be embedded in a Desarguesian projective plane $\PG(2,q)$.
\end{theorem}
\begin{proof} Suppose that $(\P,\L)$ is an antipodal plane of order $s$ embedded in $\PG(2,q)$. 
By Lemma \ref{driehoek}, we can find $A_1,A_2,A_3\in \P$ forming a triangle whose sides are lines of $\L$ and such that $A_1^\perp,A_2^\perp,A_3^\perp$ are not contained on the sides of the triangle. Since $\PGL(3,q)$ acts transitively on the triangles of $\PG(2,q)$, we can pick coordinates such that $A_1=(1,0,0),A_2=(0,1,0),A_3=(0,0,1)$.

Consider the point set $K=\{A_1,A_2,A_3,A_1^\perp,A_2^\perp, A_3^\perp\}$. Since $A_i$, $i=1,2,3$ is joined to all other points in $K$, except for $A_i^\perp$, we find that there are at most $6(s+1)-(15-3)=6s-6$ lines of $\mathcal{L}$ containing at least one point of $K$. Since $s^2+s+2>6s-6$, there is a line $\ell$ of $\L$ not through a point of $K$.

Let $\ell\cap A_2A_3=B_1,\ell\cap A_1A_3=B_2,\ell\cap A_1A_2=B_3$. The points $B_i$ are points of $\P$ since the line $\ell$ is not of the form $A_i^\perp A_j^\perp$.
The line $\ell$ contains $s-2$ points of $\P$, not lying on the sides of the triangle $A_1A_2A_3$. Denote these points by $X_j$, $j=1,\ldots, s-2$. The line $A_iX_j$ is a line of $\L$ for all $i=1,2,3$, $j=1,\ldots, s-2$ since $\ell$ does not contain any of the points $A_1^\perp,A_2^\perp, A_3^\perp$.
For each of the points $A_i$, let $T_i$ denote the product of the slopes of the $s-1$ lines of $\P$ through $A_i$, different from the sides of the triangle. 

We find that $$T_1=(A_1B_1)\prod_{j=1}^{s-2}(A_1X_j), \ T_2=(A_2B_2)\prod_{j=1}^{s-2}(A_2X_j), \text{ and }T_3=(A_3B_3)\prod_{j=1}^{s-2}(A_3X_j).$$
By Lemmas \ref{menelaos} and \ref{ceva}, we find that
\begin{align}\label{eq:ttt}
	T_1T_2T_3=-1.
\end{align}

Now consider a line $m$ of $\L$ through $A_1^\perp$ but not through $A_2,A_3,A_2^\perp$ nor $A_3^\perp$. This lines exists since $s+1>4$. Let $m\cap A_2A_3=B'_1,m\cap A_1A_3=B'_2,m\cap A_1A_2=B'_3$. Note that $B'_i$ is in $\P$ since $m$ is not of the form $A_i^\perp A_j^\perp$. Furthermore let $A_{1}^\perp=X'_1,X'_2,\ldots, X'_{s-2}$ be the $s-2$ points of $\P$ on $m$ not on the sides of the triangle. The line $m^\perp$ contains $A_{1}$ and is different from $A_{1}A_{2}$ and $A_{1}A_{3}$; let $\xi$ denote the slope of $m^\perp$. All lines of the form $A_2X'_j$ and $A_3X'_j$ are contained in $\L$ since $m$ does not contain $A_2^\perp$ nor $A_3^\perp$. 
It follows that
\[
T_1=(A_1B'_1)\xi\frac{1}{(A_1X_{1}')}\prod_{j=1}^{s-2}(A_1X'_j), \ T_2=(A_2B'_2)\prod_{j=1}^{s-2}(A_2X'_j), \text{ and }T_3=(A_3B'_3)\prod_{j=1}^{s-2}(A_3X'_j).
\]

By Lemmas \ref{menelaos} and \ref{ceva}, we find that $T_1T_2T_3=-\frac{\xi}{(A_1X_1')}$.
Using Equation \eqref{eq:ttt} we find that $\xi=(A_1X_1')$. The line with slope $\xi$ (that is, $m^\perp$) is in $\L$, but since $X'_1=A_1^\perp$, the line with slope $(A_1X_1')$ is not a line of $\L$, a contradiction. \end{proof}

\section{Code words in \texorpdfstring{$C(2,p^2)^\bot$}{C(2,p2) orth}}

Recall from the introduction that for $\c\in \C(\Pi)^\bot$,  with $\Pi$ a projective plane of order $p^2$, we have that $w(\c)\geq 2p^2-2p+2$ and that the code words meeting the lower bound are characterised. Moreover, we will see that a code word of weight $2p^2-p$ can easily be constructed if $\Pi$ admits an embedded Baer subplane. Note that all translation planes admit embedded Baer subplanes.

\begin{example}\label{excode}  Let $\Pi$ be a projective plane of order $p^{2}$ admitting a Baer subplane $\mathcal{B}$ and  let $\ell$ be a secant to $\mathcal{B}$ (a line meeting $\mathcal{B}$ in $p+1$ points, i.e.~an extended line of $\mathcal{B}$). Then the difference of the incidence vector of $\mathcal{B}$ and $\ell$ defines a code word of $\C(\Pi)^\bot$. To see this, note that $(\mathcal{B},m)=1$ for all lines $m$ of $\PG(2,p^2)$ since any line meets a Baer subplane in $1$ or $p+1$ points. Similarly, $(\ell,m)=1$ since two lines meet in $1$ or $p^{2}+1$ points (if they coincide). We conclude that $(\mathcal{B}-\ell,m)=0$ for all $m$, and hence, that $\mathcal{B}-\ell$ is a code word of $\C(\Pi)^\bot$. Since $w(\mathcal{B})=p^2+p+1$, $w(\ell)=p^2+1$, and $\B$ and $\ell$ have exactly $p+1$ $1$'s in common, we have that $\mathcal{B}-\ell$ has weight $2p^2-p$.
\end{example}

We will study code words in $\C(\Pi)^{\bot}$ with weight between $2p^2-2p+3$ and $2p^2-p$. At least for translation planes, this is sufficient to study the minimum weight code words.

\begin{notation}
	Throughout the rest of the paper we will assume that $\c$ is a code word in $\C(\Pi)^\bot$, with $\Pi$ a projective plane of order $p^{2}$, $p\geq 3$, with $w(\c)=2p^2-2p+2+\epsilon$ where $1\leq\epsilon\leq p-2$, and let $\S=\supp(\c)$. We will denote the elements of $\F_p$ by their unique representative in $\{0,...,p-1\}$, that is, the symbols used in a code word $\c$ belong to the set $\{0,...,p-1\}$.
	\par Let $\K\subseteq\{1,2,\ldots,p-1\}$ be the set of non-zero symbols used in the code word $\c$; we call these symbols the {\em colours} of $\c$. Let $K_\lambda\subseteq \S$ be the set of points that have colour $\lambda$ in $\c$. Then $\S=\cup_{\lambda \in \K}K _\lambda$.
\end{notation}

\subsection{Preliminaries}

Let $\mu(\v)$ denote the sum of all symbols in the vector $\v$ considered as integers. Since the vector $-\v$ has symbol $p-a$ precisely when a vector $\v$ has symbol $a$, we see that
\begin{align}
	\mu(\v)+\mu(-\v)=p\cdot w(\v).\label{summu}
\end{align}
Expressing that $\c$ is a code word in $\C(\Pi)^\bot$ yields that $(\c,\ell)=0$ for every line $\ell$ in $\Pi$. Let $\c|_\ell$ be the vector obtained by taking the restriction of the code word $\c$ to the line $\ell$, then this shows that
\begin{align}
	\mu(\c|_\ell)\equiv 0\pmod{p}.\label{clmod}
\end{align}

Considering Equation \eqref{clmod} for all $p^2+1$ lines through a point not in $supp(\c)$, we see that \begin{align}\mu(\c)\equiv 0\pmod{p}\label{cmod}\end{align} (and similarly $\mu(-\c)\equiv 0\pmod{p}$).

\begin{lemma}\label{2secants}
	Let $P\in \S=supp(\c)$, then the number of $2$-secants to $\S$ through $P$, say $x_P$, satisfies:
	\begin{align}
		x_P\geq 2p+1-\epsilon. \label{lowerbound}
	\end{align}
\end{lemma}
\begin{proof} There are no tangent lines to $\S$: if $m$ would be the incidence vector of a tangent line to $\S$ in the point $P$ then $(m,\c)$ is the symbol at position $P$, a contradiction since $(m,\c)=0$ for all $m$. So if we let $x_P$ denote the number of $2$-secants to $\S$ through $P\in \S$, then there are $p^2+1-x_P$ lines through $P$ with at least $2$ points of $\S$, different from $P$. Counting the points of $\S$ on lines through $P$, we find that $$2(p^2+1-x_P)+x_P+1\leq |\S|,$$ so $x_P\geq 2p-\epsilon+1$.
\end{proof}

\begin{corollary}\label{cor:even}
	The number of colours in $\K$ is even.
\end{corollary}
\begin{proof} By Lemma \ref{2secants}, there are at least $2p-\epsilon+1>0$ lines $\ell$ through a point $P\in K_\lambda$ with $|\ell\cap \S|=2$. If $\ell\cap \S=\{P,Q\}$, then $Q$ necessarily has colour $p-\lambda$. Since $p$ is odd, $p-\mu\neq \mu$ for all $\mu\in\{1,2,\ldots,p-1\}$, so the colours come in pairs of the form $\{\mu,p-\mu\}$.\end{proof}

\begin{lemma}\label{boundmu}$$|\mu(\c)-\mu(-\c)|\leq\epsilon p\leq p^2-2p.$$
\end{lemma}
\begin{proof} By possibly taking a suitable scalar multiple of $\c$, we may assume without loss of generality that $1$ occurs in $\c$. Note that then $p-1$ also occurs in $\c$. Let $P\in K_1$ and let $\ell$ be a line through $P$. Since $\mu(\c|_\ell)>0$, Equation \eqref{clmod} implies that $\mu(\c|_\ell)\geq p$. Looking at $\c|_\ell$ for all lines $\ell$ through $P$, we see that $\mu(\c)\geq(p^2+1)(p-1)+1=p^3-p^2+p$. The same argument holds for $-\c$ since 1 occurs in it, so $\mu(\c)+\mu(-\c)\geq 2(p^3-p^2+p)$. Since the sum of $\mu(\c)$ and $\mu(-\c)$ is $p. w(\c)$, and $w(\c)=2p^2-2p+2+\epsilon$, we find that $\mu(\c)+\mu(-\c)=2p^3-2p^2+2p+\epsilon p$. Hence, the difference between $\mu(\c)$ and $\mu(-\c)$ is at most $2p^3-2p^2+2p+\epsilon p-2(p^3-p^2+p)=\epsilon p$. The second inequality follows since $\epsilon\leq p-2$.
\end{proof} 

\subsection{Code words with two colours}

In this subsection we will investigate code words that have only two different colours. By possibly taking a suitable scalar multiple of $\c$, we may assume without loss of generality that $1$ occurs in $\c$, so the two colours are $1$ and $p-1$. This assumption is made throughout this subsection.

\begin{lemma}\label{0ofp} If $K_\lambda=\emptyset$ for all $\lambda\notin\{ 1,p-1\}$, then  $||K_1|-|K_{p-1}||\in \{0,p\}$. Moreover, if $||K_1|-|K_{p-1}||=p$, then $\epsilon=p-2$.

\end{lemma}
\begin{proof} Recall from Equation \eqref{cmod} that $\mu(\c)\equiv 0\pmod{p}$. We see that  $|K_1|+(p-1)|K_{p-1}|=\mu(\c)$ so $|K_1|-|K_{p-1}|\equiv \mu(\c)\equiv 0\pmod{p}$.
Since $|K_1|+(p-1)|K_{p-1}|=\mu(\c)$ and $(p-1)|K_1|+|K_{p-1}|=\mu(-\c)$, we have $\mu(\c)-\mu(-\c)=(2-p)(|K_1|-|K_{p-1}|)$. As $|\mu(\c)-\mu(-\c)|\leq \epsilon p$ by Lemma \ref{boundmu}, we have that $||K_1|-|K_{p-1}||\leq \frac{\epsilon p}{p-2}$. Since $\epsilon\leq p-2$ and $|K_1|-|K_{p-1}|\equiv 0\pmod{p}$, it follows that $||K_1|-|K_{p-1}||\in \{0,p\}$, and that $||K_1|-|K_{p-1}||=p$ is only possible if $\epsilon=p-2$.
\end{proof}

\begin{theorem}\label{isbaer}
	If $||K_1|-|K_{p-1}||=p$ then $\Pi$ admits a Baer subplane and $\c$ is given by Example \ref{excode}. \end{theorem}
\begin{proof}

	From Lemma \ref{0ofp} it follows that $\epsilon=p-2$, so $|K_1|+|K_{p-1}|=2p^2-p$. We may suppose without loss of generality that $|K_1|\geq |K_{p-1}|$, so we have $|K_1|=p^2$ and $|K_{p-1}|=p^2-p$. Through a point $P$ of $K_{1}$ there are at least $p+1$ lines not containing a point of $K_{p-1}$, and these lines must contain at least $p-1$ points of $K_{1}$ different from $P$. It follows that $P$ lies on exactly $p+1$ lines with $p-1$ points of $K_1\setminus\{P\}$, and $p^2-p$ lines containing exactly one point of $K_{p-1}$.
	\par Now consider a point $Q\in K_{p-1}$. A line through $Q$ contains at most $1$ point of $K_1$, and if it does, it contains no other points of $\S$. Since there are $p^2$ points in $K_1$, it follows that there is a unique line $\ell$ through $Q$ containing all points of $K_{p-1}$. We also see that $\ell$ does not contain points of $K_1$. Let $\A$ be the set of $p+1$ points on $\ell$ that are not in $K_{p-1}$.
	\par For any point in $K_{1}$ the $p+1$ lines through it without points of $K_{p-1}$ meet $\ell$ in a point of $\A$ and contain $p$ points of $K_1$. For any point in $\A$ there are $p$ lines through it containing $p$ points of $K_{1}$. So, $K_1\cup\A$ is a set of $p^2+p+1$ points such that every point lies on $p+1$ lines with $p+1$ points in $K_1\cup\A$, and two points of $K_1\cup\A$ determine a unique line with $p+1$ points in this set. We see that $K_1\cup\A$ forms a projective plane of order $p$ embedded in $\Pi$, hence it is a Baer subplane $\B$.
	\par Note that $\ell$ is a secant line to $\B$, and the points of $K_{p-1}$ are precisely those points on this line that are not contained in $\B$. Hence, $\c$ is the difference of the incidence vectors of $\B$ and $\ell$ as in Example \ref{excode}.
\end{proof}

\begin{lemma}
	If $|K_1|=|K_{p-1}|$, then there is no line through a point of $K_1$ that contains $p$ points of $K_{p-1}$. Likewise, there is no line through a point of $K_{p-1}$ that contains $p$ points of $K_1$. 
\end{lemma}
\begin{proof} If $|K_1|=|K_{p-1}|$, then $\mu(\c)=\mu(-\c)$, and hence, by Equation \eqref{summu}, $\mu(\c)=\mu(-\c)=\frac{p}{2}w(\c)$. It follows that $\mu(\c)=\mu(-\c)=p^3-p^2+p+\frac{\epsilon p}{2}\leq p^3-p^2+p+\frac{p(p-2)}{2}$.

Now suppose that that there is a line through a point of $K_1$ with $p$ points of $K_{p-1}$, then $\mu(\c)$ is at least  $p^2(p-1)+p(p-1)+1$, a contradiction since $p^3-p+1>p^3-p^2+p+\frac{p(p-2)}{2}$. Furthermore, if there were a line through a point of $K_{p-1}$ that contains $p$ points of $K_1$, then the code word $-\c$ has a line through a point with colour $1$ containing $p$ points with colour $p-1$, leading to a contradiction in the same way.
\end{proof}

\begin{corollary}\label{gevolg} 
	If $|K_1|=|K_{p-1}|$, then a line through a point of $K_1$ and a point of $K_{p-1}$ cannot contain $p$ points of $K_1$ nor $p$ points of $K_{p-1}$.
\end{corollary}

\begin{theorem}\label{isantipodal} If $|K_1|=|K_{p-1}|=p^2-p+2$, then $K_1$ forms an antipodal plane of order $p-1$ embedded in $\Pi$.
\end{theorem}
\begin{proof}  
We have $\mu(\c)=|K_{1}|+(p-1)|K_{p-1}|=p^3-p^2+2p=(p-1)|K_{1}|+|K_{p-1}|=\mu(-\c)$.  Recall that every line $\ell$ through a point $P\in K_1$ has $\mu(\ell)$ a multiple of $p$. This means that there are $p^2$ lines $m$ through $P$ with $\mu(\c|_m)=p$ and exactly one line, say $n$ with $\mu(n)=2p$.

A line $m$ with $\mu(m)=p$ is either a $p$-secant to $K_1$ (having $p$ points with symbol 1) and disjoint from $K_{p-1}$, or a $2$-secant to $supp(\c)$, consisting of one point of $K_1$ and one point of $K_{p-1}$. 
Since $|K_{p-1}|=p^2-p+2$, and the line $n$ contains at most $2$ points of $K_{p-1}$, we know that at least $p^2-p-3$ lines through $P$ are tangent to $K_1$.
The line $n$ with $\mu(n)=2p$ either consists of $2p$ points of $K_1$ or has at least one point of $supp(\c)$ with symbol $p-1$.
From Corollary \ref{gevolg}, we have that in the latter case, the line $n$ contains $2$ points of $K_1$ and $2$ points of $K_{p-1}$. 

Now assume that the line $n$ has $2p$ points of $K_1$ and consider a point $Q$ of $K_1$, not on $n$; such a point exists since $p>2$. Then there are at least $2p$ lines through $Q$ containing at least one point of $K_1$, different from $Q$, a contradiction since a point of $K_1$ lies on at least $p^2-p-3$ tangent lines to $K_1$ and $p\geq5$.

So that means that through every point $P$ with symbol $1$ there is a $2$-secant to $K_1$, meeting $supp(\c)$ in $4$ points: $2$ points of $K_1$ and $2$ of $K_{p-1}$. The other $p^2-p$ points of $K_1$ lie on $p$ lines through $P$ each containing $p$ points of $K_1$. So $K_1$ is a
the set of $p^2-p+2$ points forming an antipodal plane embedded in $\Pi$.
\end{proof}

\begin{corollary}\label{2coloursnieuw}
	Let $\Pi$ be a projective plane of order $p^2$, $p$ prime, $p\geq 3$. Let $\c$ be a code word of $\C(\Pi)^\bot$ using only two symbols, with $2p^2-2p+2\leq w(\c)\leq 2p^2-p$. Then either 
		\begin{itemize}
		\item $w(\c)=2p^2-2p+2$ and $\c=\lambda(\v_1-\v_2)$ with $\lambda\in\F_{p}^{*}$ and $\v_1,\v_2$ the incidence vectors of two disjoint projective planes of order $p-1$ embedded in $\Pi$; or
		\item $w(\c)=2p^2-2p+4$ and $\c=\lambda(\v_1-\v_2)$ with $\lambda\in\F_{p}^{*}$ and $\v_1,\v_2$ the incidence vectors of two disjoint antipodal planes of order $p-1$ embedded in $\Pi$; or
		\item $2p^2-2p+6\leq w(\c)\leq 2p^2-p-1$, $w(\c)$ is even, and $\c=\lambda(\v_1-\v_2)$ with $\lambda\in\F_{p}^{*}$ and $\v_1,\v_2$ the incidence vectors of two sets of size $\frac{w(\c)}{2}$; or
		\item $w(\c)=2p^2-p$, $\Pi$ contains an embedded subplane of order $p$ (i.e., a Baer subplane), and $\c$ is given by Example \ref{excode}.
	\end{itemize}
\end{corollary}
\begin{proof}
	It follows from Lemma \ref{0ofp} that $\c$ has either $|K_1|=|K_{p-1}|$, or $||K_1|-|K_{p-1}||=p$ and $\c$ has weight $2p^2-p$. In the latter case the result follows from Theorem \ref{isbaer}. In the former case $w(\c)$ is even. It follows from Result \ref{bag} and the discussion thereafter that $\c=\v_1-\v_2$ with $\v_1,\v_2$ the incidence vectors of two disjoint projective planes of order $p-1$ embedded in $\Pi$ if $w(\c)=2p^2-2p+2$. If $w(\c)=2p^2-2p+4$, Theorem \ref{isantipodal} shows that $K_1$ defines an antipodal plane of order $p-1$ embedded in $\Pi$. Replacing $\c$ by $-\c$ we find that also $K_{p-1}$ is an antipodal plane of order $p-1$ embedded in $\Pi$. Hence, $\c=\v_1-\v_2$ with $\v_1,\v_2$ the incidence vectors of two disjoint antipodal planes of order $p-1$ embedded in $\Pi$. 
\end{proof}

\subsection{Secants to the support}

In this subsection we prove some results on the secants of $\S$, the support of $\c$. Recall that  $w(\c)=2p^2-2p+2+\epsilon$, with $1\leq \epsilon\leq p-2$.

\begin{definition}
	For $A\in \S$ let $x_A$, $y_A$ and $z_A$, denote the number of 2-secants, 3-secants and 4-secants to $\S$ through $A$, respectively.
\end{definition}

\begin{lemma}\label{lem:nrsecants}
	For any $A\in \S$ we have
	\begin{align}
		2x_A+y_A&\geq p^2+2p+2-\epsilon \label{eq1}\\
		3x_A+2y_A+z_A&\geq 2p^2+2p+3-\epsilon.\label{eq2}
	\end{align}
	Denoting the set of all lines through $A$ containing at least $4$ points of $\S$ by $L_i$, $i=1,\ldots, m$, we find that
	\begin{align}
		x_A=2p+1-\epsilon+\sum_{i=1}^m(|L_i\cap \S|-3).\label{eq3}
	\end{align}
\end{lemma}
\begin{proof}
Counting the number of points of $\S$ on lines through $A$ we find:
\begin{align*}
x_A+2y_A+\sum_{i=1}^m(|L_i\cap \S|-1)+1&= |\S|.
\end{align*}
Since every line $L_i$ contains at least $4$ points of $\S$, we find
$$x_A+2y_A+3(p^2+1-x_A-y_A)+1\leq |\S|=2p^2-2p+2+\epsilon$$ which leads to Equation \eqref{eq1}. Furthermore, since there are $z_A$ lines $L_i$ with exactly $4$ points of $\S$, the other lines each contain at least $5$ points of $\S$. We find that
$$x_A+2y_A+3z_A+4(p^2+1-x_A-y_A-z_A)+1\leq |\S|=2p^2-2p+2+\epsilon$$ which leads to Equation \eqref{eq2}.
Finally, using $y_A=p^2+1-x_A-m$ we find that 
$$|\S|-1= x_A+2y_A+\sum_{i=1}^m(|L_i\cap \S|-1)=x_A+2(p^2+1-x_A-m)+\sum_{i=1}^m(|L_i\cap \S|-1).$$
Substituting $|\S|=2p^2-2p+2+\epsilon$ leads to Equation \eqref{eq3}. 
\end{proof}

Note that Equation \eqref{eq3} generalises Lemma \ref{2secants}.

\begin{lemma}\label{KvsX}The following hold:
\begin{itemize}
\item[(i)]  For every point $A\in K_\lambda$, $x_A\leq |K_{p-\lambda}|$. 
\item[(ii)] If $|K_{p-\lambda}|=2p+1-\epsilon$, then for all points $P$ in $K_{\lambda}$, $x_P=2p+1-\epsilon$. 
\item[(iii)] If $x_P=2p+1-\epsilon$ for a point $P\in K_\lambda$, then $|K_{p-\lambda}|=2p+1-\epsilon$, $P$ lies only on $2$-secants and $3$-secants to $\S$ and the points of $K_{p-\lambda}$ are precisely those on $2$-secants through $P$ to $\S$. 
\item[(iv)]  If $|K_{p-\lambda}|=2p+2-\epsilon$, then for all points $P$ in $K_{\lambda}$, $x_P= 2p+2-\epsilon$, and $P$ lies on $2p+2-\epsilon$ $2$-secants to $\S$, $p^2-2p-2+\epsilon$ $3$-secants to $\S$ and one $4$-secant to $\S$.
\end{itemize}
\end{lemma}
\begin{proof} Every $2$-secant to $\S$ through $A$ in $K_\lambda$ gives rise to a point with colour $p-\lambda$, so there are at least $x_A$ points with colour $p-\lambda$. This proves (i).

From Lemma \ref{2secants}, we have that $x_P\geq 2p+1-\epsilon$. Assuming that $|K_{p-\lambda}|=2p+1-\epsilon$ then implies that $2p+1-\epsilon=|K_{p-\lambda}|\geq x_P\geq 2p+1-\epsilon$, so $x_P=2p+1-\epsilon$, showing (ii).

Vice versa, assume that $x_P=2p+1-\epsilon$ and $|K_{p-\lambda}|\geq 2p+2-\epsilon$. Since $x_P= 2p+1-\epsilon$, it follows from Equation \eqref{eq3} that all lines through $P$ are $2$-secants or $3$-secants to $\S$. Hence, if $|K_{p-\lambda}|\geq 2p+2-\epsilon$ there is at least one point $R$ of $K_{p-\lambda}$ which does not lie on a $2$-secant to $\S$ through $P$. But that means that $PR$ would need to be a $3$-secant to $\S$, a contradiction since $P$ has colour $\lambda$ and $R$ has colour $p-\lambda$. 

Finally, let $|K_{p-\lambda}|=2p+2-\epsilon$. We know that $x_P\geq 2p+1-\epsilon$ and by the previous argument, it is impossible that $x_P=2p+1-\epsilon$ and $|K_{p-\lambda}|\neq 2p+1-\epsilon$. This means that $x_P\geq 2p+2-\epsilon$. Since $2p+2-\epsilon=|K_{p-\lambda}|\geq x_P\geq 2p+2-\epsilon$, we find that $x_P= 2p+2-\epsilon$. Since $x_P=2p+2-\epsilon$,  it follows from Equation \eqref{eq3} that $P$ lies on $p^2-2p-2+\epsilon$ $3$-secants to $\S$ and exactly one $4$-secant to $\S$. 
\end{proof}

\subsection{Code words with all colours, large \texorpdfstring{$p$}{p}}

In order to prove the main theorem we will show that if $\epsilon\in\{1,2\}$, then $|\K|<p-1$. In other words, $\c$ does not use all colours. In this subsection we take care of of the general case. Some small values of $p$ need a separate treatment and are discussed later (see Subsection \ref{sec:smallp}).

\begin{lemma}\label{oppositesnotminimal}
	Let $p>3$. There is no $\lambda \in \K$ such that $|K_\lambda|=|K_{p-\lambda}|=2p+1-\epsilon$. 
\end{lemma}
\begin{proof}
	Let $Q$ be a point of $\S$ such that $x_Q$, the number of $2$-secants through $Q$ to $\S$, is minimal, i.e. for all $P\in \S$, $x_P\geq x_Q$. Let $\mu\in \K$. Since $p-\mu\in \K$ too, we can consider a point $P$ of $K_{p-\mu}$.  By Lemma \ref{KvsX}(i), $|K_{\mu}|\geq x_P$ and $x_P\geq x_Q$ by our assumption, so
	\begin{align}
		|K_\mu|\geq x_Q \text{\ for all\ }\mu\in \K. \label{everycolourxq1}
	\end{align}	
	It follows from Lemma \ref{2secants} and Equation \eqref{everycolourxq1} that $|K_{\lambda}|\geq 2p+1-\epsilon$ for all $\lambda\in \K$.	
	\par Assume that for some colour $\lambda$, we have that $|K_\lambda|=|K_{p-\lambda}|=2p+1-\epsilon$. Without loss of generality, we can take $\lambda\leq (p-1)/2$. Consider a point $A$ in $K_\lambda$. By Lemma \ref{KvsX}(ii), we see that $x_A=2p+1-\epsilon=|K_{p-\lambda}|$
	and by Lemma \ref{KvsX}(iii), $A$ only lies on $2$-secants and $3$-secants to $\S$. Since it is impossible that $3\lambda\equiv 0\pmod{p}$, every point of $K_\lambda$, different from $A$, lies on a different line through $A$ (which is necessarily a $3$-secant to $\S$). 
	\par Let $\mathcal{L}$ be the set of those $3$-secants through $A$ that contain a point of $K_\lambda$, different from $A$. There are precisely $|K_\lambda|-1=x_A-1$ lines in $\mathcal{L}$. On each of the lines of $\L$, the third point is contained in $K_{p-2\lambda}$, so we find exactly $x_A-1$ points of $K_{p-2\lambda}$ contained in lines of $\L$. 
	\par By Equation \eqref{everycolourxq1}, $|K_{p-2\lambda}|\geq x_Q$. Since $x_Q\geq 2p+1-\epsilon$ from Lemma \ref{2secants}, and $x_A=2p+1-\epsilon\geq x_{Q}$, we have that $|K_{p-2\lambda}|\geq x_A$. Hence, there is at least one point, say $R$, of $K_{p-2\lambda}$ which is not on a line of $\mathcal{L}$. 
	But we have seen that all lines through $A$ are either $2$-secants or $3$-secants, and all points on $2$-secants through $A$ belong to $\K_{p-\lambda}$. So we find that the line $PR$ is a $3$-secant, which forces the third point on $AR$ different from $A$ and $R$ to be contained in $K_\lambda$, a contradiction since all points of $K_\lambda$ are lying on the lines of $\mathcal{L}$. We conclude that it is impossible that $|K_\lambda|=|K_{p-\lambda}|=2p+1-\epsilon$.
\end{proof}

\begin{lemma}\label{lem:notallcolours1}
	If $p\geq 11$ and $\epsilon=1$, then $|\K|<p-1$, that is, $\c$ does not use all colours.
\end{lemma}
\begin{proof}
		 Assume that $|\K|=p-1$. Since $|K_\mu|\geq 2p$ for all $\mu\in \K$ by Lemma \ref{2secants}, and at most half of the $p-1$ colours occur exactly $2p$ times by Lemma \ref{oppositesnotminimal}, it follows that
	\begin{align*}
		\frac{p-1}{2}\cdot 2p+\frac{p-1}{2}\cdot(2p+1)\leq |\S|=2p^2-2p+3\quad \Leftrightarrow \quad 3\geq \frac{p-1}{2}.
	\end{align*}
	For $\epsilon=1$ and $p>7$, this is a contradiction.
\end{proof}

For $\epsilon=1$ and $p=7$, this result will be proven in Lemma \ref{lem:eps1p7}.

\begin{lemma}\label{lem:atmostone2p-1}
	If $\epsilon=2$, then there is at most one colour $\mu$ such that $|K_{\mu}|=2p-1$.
\end{lemma}
\begin{proof}
	It follows from Lemma \ref{2secants} and Lemma \ref{KvsX}(i) that each colour occurs at least $2p-1$ times. Assume that there are two colours that occur precisely $2p-1$ times. By scaling $\c$ we may assume without loss of generality that one of these colours is 1; let $p-\lambda$ be the other, $\lambda\neq p-1$. By Lemma \ref{oppositesnotminimal} we know that $\lambda\neq1$. From Lemma \ref{KvsX}(ii) it follows that $x_{A}=2p-1$ for any point $A\in K_{\lambda}\cup K_{p-1}$. Furthermore, from Lemma \ref{KvsX}(iii) it follows that any point $A\in K_{\lambda}\cup K_{p-1}$ lies on $2p-1$ $2$-secants and $p^2-2p+2$ $3$-secants (and no other lines).
	\par Let $Q$ be a point in $K_{p-1}$. A $3$-secant through $Q$ and a point of $K_{\mu}$, $\mu\neq1$, has the third point contained in $K_{p-\mu+1}$. We conclude that
	\begin{align}
		|K_\mu|=|K_{p-\mu+1}| \text{ for\ } \mu\neq 1,p-1,2. \label{gelijk2}
	\end{align}
	Let $R$ be a point of $K_{\lambda}$. We distinguish between two cases. 
	\begin{itemize}
		\item First we look at the case $\lambda\neq p-2$. Since all lines through $R$ are $2$-secants or $3$-secants, every line through $R$ containing a point of $K_1$ contains precisely one point of $K_{p-\lambda-1}$ and vice versa. We conclude that $2p-1=|K_{1}|=|K_{p-\lambda-1}|$. We also know that $|K_{p-\lambda}|=|K_{p-(p-\lambda)+1}|=|K_{\lambda+1}|$ using Equation \eqref{gelijk2}. Furthermore, we have assumed that $|K_{p-\lambda}|=2p-1$. So, we find $|K_{\lambda+1}|=2p-1=|K_{p-\lambda-1}|$, contradicting Lemma \ref{oppositesnotminimal}.
		\item In case $\lambda=p-2$ any line through $R$ containing a point of $K_{1}$ is a 3-secant and contains exactly two points of $K_{1}$. Hence $|K_{1}|$ must be even, a contradiction since we have assumed that $|K_{1}|=2p-1$. 
	\end{itemize}
	In both cases we found a contradiction, concluding the proof.
\end{proof}

\begin{lemma}\label{lem:oppositpairs2p}
	Suppose that $\epsilon=2$, $|\mathcal{K}|=p-1$ and $|K_\mu|\geq 2p$ for every colour $\mu$. Let $\mathcal{M}$ be the multiset of colours where the multiplicity of the colour is $k$ if $|K_\mu|=2p+k$ (so only colours occuring more than $2p$ times appear in $\mathcal{M}$). Then $|\mathcal{M}|=4$.	  
	\par If moreover, there is a pair $\{\lambda,p-\lambda\}$ of colours such that  $|K_\lambda|=|K_{p-\lambda}|=2p$, then
	\begin{enumerate}[label=(\roman*)]
		\item through any point of $K_{\lambda}$ there is a unique 4-secant containing a point of $K_{p-2\lambda}$ and two points with either differing colours, both contained in $\mathcal{M}$, or two points of the same colour if that colour has multiplicity at least two in $\mathcal{M}$,
		\item the multiset $\mathcal{M}$ can be partitioned in two sets of size 2, such that the elements in one set sum to $\lambda\pmod p$  and the elements in the other set sum to $p-\lambda \pmod p$, and
		\item $|K_{2\lambda}|\geq2p+1$, i.e.~$2\lambda\in\mathcal{M}$.
	\end{enumerate}
\end{lemma}
\begin{proof}
	We find immediately that	\begin{align*}
		|\mathcal{M}|=\sum^{p-1}_{\mu=1}(|K_{\mu}|-2p)= |\S|-2p(p-1)=2p^2-2p+4-2p(p-1)=4.
	\end{align*}
	Hence, there are four (not necessarily distinct) colours $\mu_1,\dots,\mu_4$ such that $|K_{\mu_i}|\geq 2p+1$. $\mathcal{M}$ is the multiset $\{\mu_1,\mu_2,\mu_3,\mu_4\}$ containing each colour $\mu_i$ exactly  $|K_{\mu_i}|-2p$ times.
	\par Now, let $\{\lambda,p-\lambda\}$ be a pair of colours such that $|K_\lambda|=|K_{p-\lambda}|=2p$, and let $R$ be a point of $K_{\lambda}$. It follows from Lemma \ref{KvsX}(iv) that $R$ lies on $2p$ 2-secants, $p^2-2p$ 3-secants and a 4-secant $m$. The third point on a $3$-secant through $R$ and a point of $K_\lambda\setminus\{R\}$ belongs to $K_{p-2\lambda}$. Note that $p-2\lambda\neq\lambda$. Since $|K_{p-2\lambda}|\geq 2p$ and there are $2p-1$ points in $K_\lambda\setminus\{R\}$, there is at least one point, say $R'$, of $K_{p-2\lambda}$ on $m$. It follows that $m$ does not contain a point of $K_{\lambda}\setminus\{R\}$, and that there are precisely $2p-1$ 3-secants through $R$ with a point of $K_\lambda\setminus\{R\}$. Let $a$ and $b$ be the colours of the points on $m$ different from $R$ and $R'$. We already mentioned that $a\neq\lambda\neq b$ and it is also clear that $a+b\equiv\lambda\pmod{p}$.
	\par Any 3-secant through $R$ containing a point of $K_{a}$ also contains a point of $K_{p-\lambda-a}$. Note that $b\neq p-\lambda-a$. So, if $a\neq p-\lambda-a$, then $|K_{a}|\geq 2p+1$ since $|K_{p-\lambda-a}|\geq 2p$. If $a=p-\lambda-a$, then $|K_{a}|\geq 2p$ is odd, since the points of $K_a$, different from $R$ occur in pairs on the lines through $R$, different so clearly $|K_{a}|\geq 2p+1$. We can conclude that $a$ (and similarly $b$) is a colour from $\mathcal{M}$, and hence that there are colours in $\mathcal{M}$ whose sum is $\lambda$. Moreover, we have
	\begin{align*}
		0\equiv\sum^{p-1}_{\mu=1}\mu|K_{\mu}|\equiv\sum^{p-1}_{\mu=1}(|K_{\mu}|-2p)\equiv\sum^{4}_{i=1}\mu_{i}\pmod{p}
	\end{align*}
	since $\c$ is a code word. And so, the sum of the two other colours in $\mathcal{M}$ is $p-\lambda \pmod p$.
	\par So, we find a partition of  $\mathcal{M}$ in two sets of two elements such that the elements in one set sum to $\lambda\pmod p$ and those in the other set sum to $p-\lambda\pmod p$.
	\par Finally, suppose that $|K_{2\lambda}|=2p$. Then there is a unique 4-secant through $R'$ by Lemma \ref{KvsX}(iv), namely $m$. A $3$-secant through $R'$ contains either zero or two points of $K_{\lambda}$, and there is exactly one point of $K_{\lambda}$ on $m$, namely $R$. Hence $|K_{\lambda}|$ needs to be odd, a contradiction since $|K_{\lambda}|=2p$. It follows that $2\lambda\in \mathcal{M}$.
\end{proof}

\begin{lemma}\label{lem:prenotallcolours}
	Let $X=\min\{x_{A}\mid A\in\S\}$. Assume $\epsilon=2$.
	\begin{enumerate}[label=(\roman*)]
		\item If $|\K|=p-1$ and $p\geq7$, then $X\in\{2p-1,2p\}$.
		\item If $X=2p-1$ and $p\geq13$, then $|\K|<p-1$, that is, $\c$ does not use all colours.
		\item If $X=2p$ and $p\geq11$, then $|\K|<p-1$, that is, $\c$ does not use all colours.
	\end{enumerate}
\end{lemma}
\begin{proof}
	It follows from Lemma \ref{2secants} and Lemma \ref{KvsX}(i) that $X\geq 2p-1$. By Lemma \ref{KvsX}(i), $|K_{\mu}|\geq X$ for all $\mu\in \K$. If $|\K|=p-1$ and $X\geq 2p+1$, then
	\begin{align*}
		(p-1)(2p+1)\leq|\K|\cdot X\leq\sum^{p-1}_{\mu=1}|K_{\mu}|\leq |\S|=2p^2-2p+4,
	\end{align*} 
	a contradiction since $p\geq7$. This proves the first part of the statement.
	\par We now discuss the cases $X=2p-1$ and $X=2p$ separately. In both cases we assume that $|\K|=p-1$ and we derive a contradiction.
	\par Suppose first that $X=2p-1$, and let $Q$ be a point of $\S$ such that $x_Q=X$, i.e. for all $P\in \S$, $x_P\geq x_Q$. Without loss of generality, we may assume that $Q$ has colour $p-1$. Then $|K_1|=2p-1$ by Lemma \ref{KvsX}(iii). By Lemma \ref{KvsX}(ii) all lines through a point of $K_{p-1}$ are $2$-secants or $3$-secants to $\S$. From Lemma \ref{lem:atmostone2p-1} we know that $|K_{\mu}|\geq2p$ for all $\mu\neq1$. Every point of $K_2$ gives rise to a point of $K_{p-1}$ on a $3$-secant through $Q$, so $|K_{p-1}|=|K_2|+1$. Since $|K_2|\geq 2p$, $|K_{p-1}|\geq 2p+1$. We find that
	\begin{align*}
		|K_{1}|+|K_{p-1}|+\sum^{p-2}_{\mu=2}(|K_{\mu}|-2p)= |\S|-2p(p-3)=4p+4\quad\Rightarrow\quad \sum^{p-2}_{\mu=2}(|K_{\mu}|-2p)\leq 4.
	\end{align*}
	Hence, there is a colour $\nu$ with $2\leq\nu\leq p-3$ such that $|K_{\nu}|=|K_{\nu+1}|=2p$, since $p\geq 13$. Let $T$ be a point in $K_{p-\nu}$. Recall that the points of $K_{p-1}$ only lie on 2-secants and 3-secants to $\S$. So, each of the $2p+1$ points of $K_{p-1}$ occurs on a $3$-secant through $T$, and all these 3-secants are different since $p-\nu\neq 2$. However, each of them contains a point of $K_{\mu+1}$, contradicting that $|K_\nu|=2p$. This proves the second part of the statement.	
	\par Suppose now that $X=2p$. Recall that $|K_{\mu}|\geq X$ for all $\mu\in \K$,
	so we can apply Lemma \ref{lem:oppositpairs2p}. Let $\mathcal{M}$ be the multiset introduced in the statement of this lemma. Since $|\mathcal{M}|=4$ and $p\geq11$, we can find a pair $\{\lambda,p-\lambda\}$ of opposite colours such that  $|K_\lambda|=|K_{p-\lambda}|=2p$. By Lemma \ref{lem:oppositpairs2p}(iii), we find that $2\lambda\in\mathcal{M}$ and (switching the roles of $\lambda$ and $p-\lambda$) also $p-2\lambda\in\mathcal{M}$. Now applying Lemma \ref{lem:oppositpairs2p}(ii) we find that $3\lambda,p-3\lambda\in\mathcal{M}$. So, $\mathcal{M}=\{\lambda,3\lambda,p-3\lambda,p-\lambda\}$. Since $p\geq 11$, we have that $5\lambda\notin\mathcal{M}$ and $p-5\lambda\notin\mathcal{M}$, i.e.~$|K_{5\lambda}|=|K_{p-5\lambda}|=2p$. But since no two elements of $\mathcal{M}$ sum to $5\lambda\pmod{p}$, we find a contradiction by Lemma \ref{lem:oppositpairs2p}(ii). This contradiction finishes the proof. 
\end{proof}

\begin{lemma}\label{lem:notallcolours2}
	If $p\geq 13$ and $\epsilon=2$, then $|\K|<p-1$, that is, $\c$ does not use all colours.
\end{lemma}
\begin{proof}
	Assume that $|\K|=p-1$, then it follows from Lemma \ref{lem:prenotallcolours}(i) that $\min\{x_{A}\mid A\in\S\}$ equals $2p-1$ or $2p$. For the former case we find a contradiction through Lemma \ref{lem:prenotallcolours}(ii), for the latter through Lemma \ref{lem:prenotallcolours}(iii).
\end{proof}

For $\epsilon=2$ and $p=7,11$, this result will be proven in Lemmas \ref{lem:eps2p7} and \ref{lem:eps2p11}.

\subsection{Code words with all colours, small \texorpdfstring{$p$}{p}}\label{sec:smallp}

In Lemma \ref{lem:notallcolours1}, we considered the case that all colours were used, for $\epsilon=1$. We found a contradiction if $p>7$. We now show this result for $p=7$.

\begin{lemma}\label{lem:eps1p7}
	If $p=7$ and $\epsilon=1$, then $|\K|<p-1=6$, that is, $\c$ does not use all colours.
\end{lemma}
\begin{proof}
	In this case $|\S|=2p^2-2p+3=87$. Assume that $|\K|=p-1=6$. It follows from Lemma \ref{2secants} that each point lies on at least 14 2-secants and thus from Lemma \ref{KvsX}(i) that each colour occurs at least 14 times. Then from Lemma \ref{oppositesnotminimal}, it follows that the $6$ colours occur in pairs $\{\lambda,7-\lambda\}$, of which one colour occurs $14$ times and the other $15$ times. Since we could replace $\c$ by $-\c$, we may assume without loss of generality that $|K_1|=14$ and $|K_6|=15$. Fix a point $Q\in K_6$; by Lemma \ref{KvsX}(ii) and (iii) there are only 2-secants and 3-secants to $\S$ through $Q$. Since a line through $Q$ contains a point of $K_6\setminus\{Q\}$ if and only if it contains a point of $K_2$, we find $|K_2|=14$ and thus also $|K_5|=15$. Since a line through a point of $K_4$ and $Q$ contains two points of $K_4$, $|K_4|$ is even, and so $|K_4|=14$, and hence, $|K_3|=15$.
	\par Using Lemma \ref{KvsX}(ii) and (iii), we find that $\S$ is a set of $87$ points of which $3\cdot 15=45$ points (those of $K_3,K_5,K_6$) each lie on $14$ $2$-secants and $36$ $3$-secants. The remaining $3\cdot 14=42$ points  belong to some $K_{\mu}$ such that $|K_{7-\mu}|=15$, so it follows from Lemma \ref{KvsX}(iv) that they each lie on $15$ $2$-secants to $\S$, $34$ $3$-secants to $\S$ and precisely one $4$-secant to $\S$. Counting incident couples $(P,L)$ where $P\in \S$ and $L$ is a $4$-secant to $\S$ then yields $45\cdot 0+42\cdot 1=4Y$, where $Y$ is the number of 4-secants to $\S$, a contradiction since $Y$ needs to be an integer.
\end{proof}

In Lemma \ref{lem:notallcolours2}, we considered the case that all colours were used, for $\epsilon=2$. We found a contradiction if $p\geq13$. We now show this result for $p\in\{7,11\}$. We will often look at the \emph{colour distribution} of the code word $\c$, which is the tuple $\left(|K_{i}|\right)^{p-1}_{i=1}$.

\begin{lemma}\label{lem:eps2p7}
	If $p=7$ and $\epsilon=2$, then $|\K|<p-1=6$, that is, $\c$ does not use all colours.
\end{lemma}
\begin{proof}
	In this case $|\S|=2p^2-2p+4=88$. Assume that $|\K|=p-1=6$. Let $Z=\min\{|K_{\mu}|\mid \mu\in\K\}$. Since $15\cdot6>88$, we have that $Z\leq 14$. From Lemmas \ref{KvsX}(i) and \ref{lem:prenotallcolours} it follows that $Z=13$ or $Z=14$. We distinguish between those cases.
	\par Firstly, we look at the case $Z=13$. Without loss of generality we can assume that $|K_1|=13$. From Lemma \ref{lem:atmostone2p-1} it follows that $|K_\mu|\geq 14$ for $\mu\neq 1$. Any line through a point of $K_6$ is a 2-secant or 3-secant by Lemma \ref{KvsX}(ii) and (iii). Considering the lines through a point of $K_{6}$, we find that $|K_6|=|K_2|+1$, $|K_3|=|K_5|$ and $|K_4|$ is even. It also follows that if we consider a point $R$ of $K_2$, then every line through $R$ and a point of $K_6$ is necessarily a $3$-secant containing precisely two points of $K_6$. Hence, $|K_6|$ must be even. 
	\par Since all colours different from $1$ occur at least $14$ times, and $|\S|=2p^2-2p+4=88$ we find that	
	\begin{enumerate}
		\item $(|K_1|,|K_2|,|K_3|,|K_4|,|K_5|,|K_6|)=(13,15,15,14,15,16)$ or
		\item $(|K_1|,|K_2|,|K_3|,|K_4|,|K_5|,|K_6|)=(13,15,14,16,14,16)$.
	\end{enumerate}
	\par In the first case, consider a point $T$ of $K_3$; then $T$ lies on 14 2-secants to $\S$, and precisely one 4-secant, say $m$, by Lemma \ref{KvsX}(iv). It follows that there are at most 13 3-secants containing a point of $K_1$, and a point of $K_3$, different from $T$. Since $|K_3|=15$, we see that there is at least one point $T'$ of $K_3$ on $m$, different from $T$. Since $|K_2|=15$, and a $3$-secant through $T$ and a point of $K_2$ contains two points of $K_2$, we find that $m$ contains a point $T''$ of $K_2$. But this implies that the point of $\S$ on $m$ different from $T,T',T''$ needs to belong to $K_6$, a contradiction since a point of $K_6$ does not lie on $4$-secants.
	\par Similarly, in the second case, consider a point $T$ of $K_4$; then $T$ lies on 14 2-secants to $\S$, and precisely one $4$-secant, say $m$, by Lemma \ref{KvsX}(iv). Since $|K_4|=16=|K_6|$, it follows that $m$ contains at least one point of $K_6$, a contradiction since a point of $K_6$ does not lie on $4$-secants.
	\par Secondly, we look at the case $Z=14$. Without loss of generality we can assume that $|K_1|=14$. Note that by Lemma \ref{KvsX}(iv), every point of $K_6$ lies on 14 2-secants, 35 3-secants and a unique 4-secant. We distinguish several cases, based on $|K_{2}|$.
	\begin{itemize}
		\item {\bf Assume first that $|K_2|=14$.} Looking the lines through a point $R\in K_5$, applying Lemma \ref{KvsX}(iv) and considering that $|K_1|$ is even, results in the fact that the unique $4$-secant through $R$ contains $3$ points with symbols from $3,4,5,6$ adding to $2\pmod{7}$. It follows that these three points are:
		\begin{enumerate} 
			\item three points of $K_3$
			\item two points of $K_5$ and one of $K_6$
			\item one point of $K_4$ and two of $K_6$.
		\end{enumerate}
		Since a $3$-secant through $R$ with a point of $K_3$ contains a point of $K_6$, it follows that in the third case $|K_6|=|K_3|+2$, and hence, $(|K_3|,|K_6|)=(14,16)$ or $(|K_3|,|K_6|)=(15,17)$. Moreover, $|K_4|=|K_5|$ in the third case. Using that $|\S|=88$, the colour distribution is given by
		\begin{enumerate}
			\item $(|K_1|,|K_2|,|K_3|,|K_4|,|K_5|,|K_6|)=(14,14,17,14,15,14)$, or
			\item $(|K_1|,|K_2|,|K_3|,|K_4|,|K_5|,|K_6|)=(14,14,14,14,17,15)$, or
			\item[3a.] $(|K_1|,|K_2|,|K_3|,|K_4|,|K_5|,|K_6|)=(14,14,14,15,15,16)$, or 
			\item[3b.] $(|K_1|,|K_2|,|K_3|,|K_4|,|K_5|,|K_6|)=(14,14,15,14,14,17)$.
		\end{enumerate}
		\par Considering the lines through a point of $K_5$ shows that a line through a point of $K_5$ and $K_1$ necessarily is a $3$-secant containing two points of $K_1$. But then, considering the lines through a point of $K_1$ shows that every point of $K_5$ gives rise to a point in $K_1$, so $|K_1|\geq|K_5|$, a contradiction in cases 1,2,3a. In case 3b the unique $4$-secant, say $m$, through a point of $K_5$ has one point of $K_4$, one of $K_5$ and two points of $K_6$. Let $P$ be a point of $K_6$ on this $4$-secant. Since $|K_1|=14$, it follows that $P$ lies on only $2$-and $3$-secants and $m$. In particular, since a $3$-secant with one point of $K_4$ through $P$ contains a second point of $K_4$, and $m$ contains precisely one point of $K_4$, $|K_4|$ is odd, a contradiction.
		\item {\bf Then assume that $|K_2|=15$.} We distinguish four subcases based on $|K_{3}|$.
		\begin{itemize}
			\item If $|K_3|=14$, then through a point $R$ of $K_4$ there is a unique $4$-secant by Lemma \ref{KvsX}(iv), and looking at the lines through $R$, we see that this 4-secant has three points of $K_4$ and one of $K_2$. So, $(|K_1|,|K_2|,|K_3|,|K_4|,|K_5|,|K_6|)=(14,15,14,17,14,14)$. We can see that the unique $4$-secant through a point of $K_2$ has three points of $K_4$ and in particular, it follows that a line through a point of $K_3$ and a point of $K_2$ is always a $3$-secant. However, looking at the lines through a point of $K_3$ then this implies that the number of points of $K_2$ is even, a contradiction since $|K_2|=15$.
			\item If $|K_3|=15$, then considering the lines through a point $Q\in K_6$ (among which a unique 4-secant $\ell$) shows that there are either 14 or 15 3-secants with a point of $K_3$ and $K_5$ through $Q$; there cannot be 13 since $6+3+3+5\not\equiv0\pmod{7}$.
			\par If there are $14$ such $3$-secants, then $|K_5|=14$. Since $|K_3|=15$, it follows that $\ell$ contains exactly one point of $K_3$. The two remaining points on $\ell$ have colours adding up to $5\pmod{7}$, and hence, it is impossible that $\ell$ contains a point of $K_2$. Since $|K_2|=15$, there are $15$ $3$-secants through $Q$ with a point of $K_2$ and a point of $K_6$. Since $|K_4|\geq 14$ and there is at most one point of $\ell$ in $K_4$, there are $7$ $3$-secants through $Q$ with two points of $K_4$. We see that $Q$ lies on $14+14+15+7=50=p^{2}+1$ lines that are 2-secants or $3$-secants, a contradiction since $Q$ lies also on the 4-secant $\ell$.
			\par If there are 15 3-secants with a point of $K_3$ and $K_5$ through $Q$, then it follows that $|K_{5}|\geq15$, and hence $|K_{4}|\leq 15$. So, there are precisely 7 3-secants through $Q$ with two points of $K_4$, and hence, $13$ $3$-secants with a point of $K_2$ and one of $K_6$. It follows that the three points of $\ell$, different from $Q$ are two points of $K_2$ and one of $K_4$. Hence, the colour distribution is $(14,15,15,15,15,14)$.
			\par A point $S$ of $K_1$ lies on a unique $4$-secant by Lemma \ref{KvsX}(iv). Given the colour distribution and looking at the lines through $S$, we can see that this 4-secant contains two points of $K_5$ and one of $K_3$. Similarly, a point of $K_6$ lies on a unique $4$-secant with two points of $K_2$ and one of $K_4$. Consider a point $R$ of $K_2$, and suppose that $R$ does not lie on a unique $4$-secant. Then $R$ lies on $15$ $2$-secants, and either on 34 3-secants, zero 4-secants and a unique 5-secant, or on 33 3-secants and two $4$-secants. Suppose first that we are in the former case, then by the reasoning above, the unique $5$-secant through $R$ cannot contain points of $K_1$ nor $K_6$, so it contains $2$ points of $K_2$, $2$ of $K_3$ and one of $K_4$. Let $T$ be a point of $K_3$ on this $5$-secant, then $T$ necessarily lies on $15$ $2$-secants, each containing a point of $K_4$. However, there is a point of $K_4$ on the $5$-secant through $T$, leading to $|K_4|\geq 16$, a contradiction.
			\par We conclude that if $R$ does not lie on a unique $4$-secant, it lies on two $4$-secants, say $m$ and $n$, on $15$ $2$-secants, and $33$ $3$-secants. Since a point of $K_1$ and a point of $K_2$ do not occur together on a $4$-secant, and $|K_4|=15,|K_1|=14$, we find that there is exactly one point of $K_4$ contained on $m$ or $n$, say on $m$. We also find that there is at least one point of $K_3$ on $m$ or $n$ since $|K_2|=|K_3|$, and it follows that this point lies on $n$ since all points of $K_5$ lie on a 2-secant through $R$. There cannot be a point of $K_6$ on $n$ since the unique $4$-secant through a point of $K_6$ does not contain a point of $K_3$. It follows that the remaining two points on $n$ can only belong to $K_2$ and $K_3$, a contradiction since these two points cannot sum to $2\pmod{7}$. We conclude that a point of $K_2$ lies on precisely one $4$-secant. The same reasoning (replacing $i$ with $7-i$ throughout) shows that every point of $K_5$ lies on a unique $4$-secant.	
			\par If a point of $K_3$ lies on a $5$-secant, then by the above argument, a point of $K_{2}$ can only be contained in a $3$-secant through this point, containing another point of $K_2$, contradicting that $|K_2|$ is odd.
			\par If a point of $K_3$ lies on two $4$-secants, say $m$ and $n$, it lies on $15$ $2$-secants, and since $|K_2|$ is odd, we find that there is a point of $K_2$ on $m$ or $n$, say on $m$. Since no $4$-secant through a point of $K_1$ contains a point of $K_2$, and no $4$-secant through a point of $K_6$ contains a point of $K_3$, we find that $m$ does not contain a point of $K_1$, nor $K_6$, and also not $K_4$ since all points of $K_4$ lie on $2$-secants through a point of $K_3$. It follows that $m$ needs to contain two points with colours from $2,3$ or $5$ adding up to $9$, a contradiction. The same reasoning (replacing $i$ with $7-i$ throughout) shows that every point of $K_4$ lies on a unique $4$-secant. We conclude that every point of  $\S$ lies on precisely one $4$-secant to $\S$. It follows that each point lies on $35$ $3$-secants and $14$ $2$-secants. But then $|\S|35=3f$ where $f$ is the number of $3$-secants, a contradiction since $|S|=88$ and $88$ nor $35$ are a multiple of $3$.
			\item If $|K_3|=16$, then applying Lemma \ref{KvsX}(iv) and considering the lines though a point of $K_6$ shows that the unique $4$-secant necessarily has one point of $K_2$ and two points of $K_3$. Hence, the colour distribution is $(14,15,16,14,14,15)$. Looking at the lines through a point of $K_2$ shows that the unique $4$-secant to it contains two points of $K_3$ and one point of $K_6$. In particular, a line through a point of $K_1$ and $K_2$ is a $3$-secant. Looking at the lines through a point $R$ of $K_1$, we see that, since $|K_2|=15$ and $|K_4|=14$, not all points of $K_2$ lie on $3$-secants through $R$, a contradiction.
			\item If $|K_3|=17$ then the colour distribution is $(14,15,17,14,14,14)$. Considering the lines through a point $R$ of $K_3$ and applying Lemma \ref{KvsX}(iv) shows that the unique $4$-secant through it needs to contain $3$ points of $K_3$. But then the fourth point belongs to $K_5$, a contradiction since $|K_5|=|K_6|=14$, and every line through $R$ and a point of $K_6$ is a $3$-secant containing a point of $K_5$.
		\end{itemize}
		\item {\bf Now assume that $|K_2|=16$.} Considering the lines through a point of $K_6$, and applying Lemma \ref{KvsX}(iv), shows that the unique $4$-secant through it contains two points of $K_2$ and one of each of $K_4$ and $K_6$, and that the colour distribution is $(14,16,14,15,14,15)$. Considering the lines through a point of $K_2$ shows that a point of $K_1$ and $K_2$ never occur on a $4$-secant (the unique $4$-secant through a point of $K_2$ contains two points of $K_2$ and one of each of $K_4$ and $K_6$). But then, looking at the lines through a point of $K_1$ shows that this point lies on either one $5$-secant or on two $4$-secants. Since $|K_2|=|K_4|+1$, there is a point of $K_2$ contained in these, which by the above, means that it needs to be a $5$-secant. However, since $|K_5|=14$, a point of $K_2$ does not lie on a $5$-secant, a contradiction.
		\item {\bf Finally assume that $|K_2|\geq 17$.} Considering the lines through a point of $K_6$ shows that there are at least $15$ $3$-secants through it with a point of $K_2$, leading to at least $16$ points in $K_6$, a contradiction (since every colour occurs at least $14$ times, $|\S|=88$ and $|K_2|\geq 17$).
	\end{itemize}
	We find a contradiction in each of the four cases, so we find that the assumption $Z=14$ leads to a contradiction. Since $Z\in \{13,14\}$, as we mentioned in the beginning of the proof, and both cases lead to a contradiction, this finishes the proof.
	\end{proof}

\begin{lemma}\label{lem:eps2p11}
	Let $p=11$ and $|\S|=2p^2-2p+4$, then $\c$ does not use all colours.
\end{lemma}
\begin{proof}
	Assume that $|\K|=p-1$ and define $X=\min\{x_{A}\mid A\in\S\}$. By Lemma \ref{lem:prenotallcolours} we only need to deal with the case $X=2p-1=21$. In that case, there is a point $A$ with $x_A=2p-1$. Without loss of generality $A\in K_{10}$, so it follows from Lemma \ref{KvsX}(iii) that $|K_1|=2p-1$ and that $A$ lies only on $2$-secants and $3$-secants. Looking at the $3$-secants through $A$, it follows that $|K_2|+1=|K_{10}|$, $|K_6|$ is even, and
	\begin{align}
	|K_\mu|=|K_{12-\mu}| \ \text{for}\ \mu=3,4,5.\label{mus}
	\end{align}
	It follows from Lemma \ref{lem:atmostone2p-1} that every colour, different from $1$, occurs at least $2p$ times. 
		\par Suppose first that $|K_2|=2p$ and consider a point $Q\in K_9$. Then $Q$ lies on a unique $4$-secant $\ell$ by Lemma \ref{KvsX}(iv), and a $3$-secant through $Q$ and a point of $K_1$ contains a second point of $K_1$. Since $|K_1|$ is odd, it follows that $\ell$ contains exactly one point of $K_1$. The other two points on $\ell$, different from $Q$ and the point of $K_1$ have colours adding up to $12$. In particular, since every point of $K_2$ lies on a $2$-secant through $Q$, there does not lie a point of $K_{10}$ on $\ell$. Since $|K_{10}|=2p+1$, and every point of $K_{10}$ is joined to $Q$ with a $3$-secant containing a point of $K_3$, we have that $|K_3|\geq 2p+1$. By Equation \eqref{mus} this implies that $|K_9|\geq 2p+1$. We now distinguish between two cases.
	\begin{itemize}
		\item Suppose that $|K_4|=2p$ (along with $|K_2|=2p$). Consider a point $R$ of $K_7$, then $R$ lies on a unique $4$-secant $m$ by Lemma \ref{KvsX}(iv). Since every point of $K_3$ on a $3$-secant through $R$ gives rise to a point of $K_1$ on it, and $|K_1|=2p-1$ and $|K_3|\geq 2p+1$, $m$ contains $2$ points of $K_3$, and $|K_3|=2p+1$. It follows that the points on $m$ belong to $K_3$, $K_7$ and $K_9$. By considering the $3$-secants through $R$, it then follows that $|K_5|=|K_{10}|$ and $|K_7|=|K_8|+1$. We conclude, using Equation \eqref{mus} and $|\S|=2p^2-2p+4$, that the only possibility for the colour distribution in this case is $(2p-1,2p,2p+1,2p,2p+1,2p,2p+1,2p,2p+1,2p+1)$, and in particular, that $|K_8|=2p$. Let $T$ be a point of $K_3$ on $m$. Then every 3-secant through $T$ containing a point of $K_1$ contains a point of $K_7$. Since there is only one point of $K_7$ on the unique $4$-secant $m$ through $T$, we have $|K_7|=|K_1|+1=2p$, a contradiction.
		\item Suppose that $|K_4|\geq 2p+1$ (along with $|K_2|=2p$). Again, using Equation \eqref{mus} and $|\S|=2p^2-2p+4$, it follows that the only possible colour distribution is $(2p-1,2p,2p+1,2p+1,2p,2p,2p,2p+1,2p+1,2p+1)$. Consider a point $U$ of $K_5$, then, since $|K_6|=2p$, $U$ lies on a unique $4$-secant $n$ by Lemma \ref{KvsX}(iv). Since $|K_2|=2p$ and $|K_4|=2p+1$, $n$ necessarily contains a point of $K_4$, and since $|K_3|$ is odd, $n$ contains a point of $K_3$. It follows that $n$ contains a point of $K_3$, $K_4$, $K_5$ and $K_{10}$. Consider the point $V$ of $K_4$ on $n$. Since $|K_7|=2p$, $n$ is the unique $4$-secant through $V$ by Lemma \ref{KvsX}(iv). Since $n$ does not contain a point of $K_1$ nor $K_6$, the line through $V$ and a point of $K_6$ is a $3$-secant, which also contains a point of $K_1$. This means that $|K_1|=|K_6|$, a contradiction.
	\end{itemize}
	\par In both cases we obtained a contradiction, so $|K_{2}|\neq 2p$. Now, suppose that $|K_2|=2p+1$.  It follows that $|K_{10}|=2p+2$. We again distinguish between two cases.
	\begin{itemize}
		\item Suppose that $|K_3|=2p$ (along with $|K_2|=2p+1$). It follows that $|K_9|=2p$. Consider a point $R$ of $K_8$. Then, since $|K_3|=2p$, $R$ lies on a unique $4$-secant $m$ by Lemma \ref{KvsX}(iv). Since $|K_1|=2p-1$ and $|K_2|=2p+1$, $m$ contains two points of $K_2$. It follows that $m$ contains 2 points of $K_2$ and a point of each of $K_8$ and $K_{10}$. A $3$-secant through $R$ with a point of $K_4$ contains a point of $K_{10}$, and thus it follows that $|K_4|+1=|K_{10}|$. This in turns implies that the colour distribution is $(2p-1,2p+1,2p,2p+1,2p,2p,2p,2p+1,2p,2p+2)$.
		Consider then a point $S$ of $K_2$ on $m$. Since $|K_9|=2p$, $m$ is the unique $4$-secant through $S$  by Lemma \ref{KvsX}(iv). A $3$-secant through $S$ with a point of $K_4$ contains a point of $K_5$, and since $m$ does not contain points of $K_4$ nor $K_5$, it follows that $|K_4|=|K_5|$, a contradiction.
		\item Suppose that $|K_3|\geq 2p+1$ (along with $|K_2|=2p+1$). Since $|K_9|=|K_3|$ it immediately follows that the colour distribution must be $(2p-1,2p+1,2p+1,2p,2p,2p,2p,2p,2p+1,2p+2).$
		Consider a point $T$ of $K_7$. Since $|K_4|=2p$, $T$ lies on a unique $4$-secant $m$. Since $|K_7|=|K_8|=2p$, $m$ contains a point of $K_8$. Furthermore, since $|K_1|=2p-1$ and $|K_3|=2p+1$, $m$ needs to contain two points of $K_3$. But $3+3+7+8\not\equiv0\pmod{11}$, a contradiction.
	\end{itemize}
	\par In both cases we obtained a contradiction, so $|K_{2}|\neq 2p+1$. Now, suppose that $|K_2|\geq 2p+2$. This implies that $|K_{10}|\geq 2p+3$. It follows that $|K_\mu|=2p$ for $\mu\neq 1,2,10$, $|K_2|=2p+2$ and $|K_{10}|=2p+3$. Consider a point $U$ of $K_8$. Since $|K_3|=2p$, $U$ lies on a unique $4$-secant $n$ by Lemma \ref{KvsX}(iv). A $3$-secant through $U$ and a point of $K_2$ contains a point of $K_1$. But since $|K_1|=2p-1$ and $|K_2|=2p+2$, $n$ needs to contain $3$ points of $K_2$, a contradiction since $2+2+2+8\not\equiv0\pmod{11}$.
\end{proof}

\subsection{The proof of the Main Theorem}

In this subsection we will complete the proof of the Main Theorem. We may assume without loss of generality that $\c$ is scaled such that there is a $Q\in K_{p-1}$ with $x_Q=\min\{x_A\mid A\in \S\}$. We will make this assumption throughout this subsection. Recall from Corollary \ref{cor:even} that $|\K|$ is even. We define $r=\frac{|\K|}{2}$. 

\begin{definition}
	We define the graph $G$ with vertices $\{1,2,\ldots,p-1\}$, where two vertices $\alpha,\beta$ are adjacent if and only if $\alpha+\beta=p$ or $\alpha+\beta=p+1$.
\end{definition}

It's easy to see that $G$ is a path $1,p-1,2,p-2,3,\ldots,(p-1)/2,(p+1)/2$, with a loop added to the end vertex $(p+1)/2$. Note that the degree of the vertex $1$ is one, and all other vertices (including the loop vertex $(p+1)/2$) have degree two.

\begin{definition}
	The colours of the set $\K$ used in $\c$ define an induced subgraph $G'$ of $G$.
\end{definition}

\begin{lemma}\label{disconnected}
	If $\epsilon\in\{1,2\}$, the graph $G'$ has at most $2$ connected components, and if there are two connected components, the vertex $\frac{p+1}{2}$ is contained in $G'$. Furthermore, $r<p/6$.
\end{lemma}
\begin{proof}
	Assume that $G'$ is disconnected. Then there are at least $3$ vertices of degree $1$, one of which is $1$. Let $d$ be the number of vertices of degree $1$ in $G'$. Recall that if $\alpha\in \K$, then $p-\alpha\in \K$ (indeed, by Lemma \ref{2secants} every point lies on at least $2p-1$ 2-secants). So every vertex in $G'$ has at least degree $1$, i.e.~there are no isolated vertices in $G'$. Hence, if a colour $\beta$ has degree $1$ in $G'$, then $p+1-\beta\notin\K$. Consider a $3$-secant $n$ to $\S$ through $Q$. If $n$ contains a point of $K_\beta$, different from $Q$, then the third point of $\S$ on $n$ has colour $p+1-\beta$. This shows that no point on a $3$-secant to $\S$  through $Q$ has degree $1$ in $G'$.
	\par Hence, the $2y_Q$ points that are different from $Q$ and lie on a 3-secant through $Q$, all have a colour of degree different from 1. Since each colour of degree $1$ occurs at least $x_Q$ times by Lemma \ref{KvsX}(i) and the definition of $Q$, we find that there are at most $|\S|-dx_Q$ points with a colour of degree different from $1$. Hence,
	\[
		2y_Q\leq |\S|-dx_Q.
	\]
	From Equation \eqref{eq1} in Lemma \ref{lem:nrsecants} we know that $-2y_Q\leq 4x_Q-2p^2-4p-4+2\epsilon$, so we find that
	\[
	dx_Q\leq 2p^2-2p+2+\epsilon-2y_Q\leq 4x_Q+3\epsilon-6p-2,
	\]
	and finally that
	\begin{align}
		(d-4)x_Q\leq 3\epsilon-6p-2.\label{degree}
	\end{align}
	Since $\epsilon\in \{1,2\}$, the right hand side is strictly smaller than zero. We conclude that $d<4$, and hence, since $d\geq 3$, that $d=3$. We find that if $G'$ is disconnected, there are exactly two components, one of which then necessarily contains the vertex $\frac{p+1}{2}$. Furthermore, since $d=3$, Equation \eqref{degree} implies that $x_Q\geq 6p-4$. Recall from Lemma \ref{KvsX} that every colour occurs at least $x_Q$ times, so $2rx_Q\leq 2p^2-2p+4$. It follows that $r\leq \frac{p^2-p+2}{6p-4}<p/6$ since $p\geq 7$.
\end{proof}

\begin{lemma}\label{lem:connected}
	If $\epsilon\in\{1,2\}$, the graph $G'$ is connected.
\end{lemma}
\begin{proof}
	Assume that $G'$ is not connected. Then Lemma \ref{disconnected} shows that it has two connected components, one of which contains $(p+1)/2$. Let $Q$ be as before. Suppose that the other degree $1$ vertex in the component of $1$ is $\gamma$; since $\beta\in \K$ implies that $p-\beta\in \K$, $\gamma=p-t$ for some $t<\frac{p-1}{2}$, so this component is a path of length $2t$. Since $|\K|=2r$, it follows that the third degree $1$ vertex is $\alpha=\frac{p+1}{2}-r+t$.
	\par We know that no points of $K_{p-t}$ nor $K_\alpha$ occur on $3$-secants through $Q$. Counting the points on lines through $Q$, not in $K_\alpha$ nor $K_{p-t}$ yields that $1+x_Q+2y_Q\leq |\S|-2x_Q$, and hence, $3x_Q+2y_Q\leq 2p^2-2p+\epsilon+1$. From Equation \eqref{eq2} we have that $3x_Q+2y_Q+z_Q\geq 2p^2+2p+3-\epsilon$. It follows that  $2p^2-2p+\epsilon+1+z_Q\geq 3x_Q+2y_Q+z_Q\geq 2p^2+2p+3-\epsilon$, so $z_Q\geq 4p+2-2\epsilon>0$.
	\par Now suppose that each of the $z_Q$ $4$-secants through $Q$ has at least one point, different from $Q$, that does not lie in $K_\alpha\cup K_{p-t}$. Then there are at least $x_Q+2y_Q+z_Q$ points not in $K_\alpha\cup K_{p-t}$, and hence,
	$$x_Q+2y_Q+z_Q\leq |\S|-2x_Q\leq 2p^2-2p+4-2x_Q,$$
	so $3x_Q+2y_Q+z_Q\leq 2p^2-2p+4$ but Equation \eqref{eq2} shows that $3x_Q+2y_Q+z_Q\geq  2p^2+2p+1$, a contradiction.
	\par So, there is a $4$-secant through $Q$, say $n$, containing only points of $K_\alpha\cup K_{p-t}$ apart from $Q$. Suppose it contains $i\in \{0,1,2,3\}$ points of $K_{p-t}$, then the sum of colours on $n$ is
	$$p-1+i(p-t)+(3-i)\alpha$$
	which must be congruent to $0\pmod{p}$. If $i=0$, we find that $3\alpha-1=3(\frac{p+1}{2}-(r-t))-1$ must be a multiple of $p$, which, since $0\leq\alpha\leq \frac{p-1}{2}$ implies that $3(\frac{p+1}{2}-(r-t))-1=p$, leading to $6(r-t)-1=p$, a contradiction since $1\leq t<r<p/6$ by Lemma \ref{disconnected}. Similarly, when $i=1$, we find that $2r-t$ needs to be a multiple of $p$, when $i=2$, that $2(r+t)+1$ needs to be a multiple of $p$ and when $i=3$ that $3t+1$ needs to be a multiple of $p$, all of which lead to a contradiction since $1\leq t<r<p/6$.
\end{proof}

\begin{theorem}\label{final}
	There are no code words of size at most $2p^2-2p+4$, and using more than $2$ colours, in $\C(\Pi)^\bot$, where $\Pi$ is a projective plane of order $p^2$, $p\geq 7$.
\end{theorem}
\begin{proof}
Let $\c$ be a code word of $\C(\Pi)^\bot$ of size $2p^2-2p+2+\epsilon$ with $\epsilon\in \{1,2\}$ and let $\K$ be its set of colours. Assume that $|\K|>2$. From the assumption and Corollary \ref{cor:even}, it follows that $|\K|\geq 4$. We know that not all colours are used from Lemmas \ref{lem:notallcolours1}, \ref{lem:notallcolours2}, \ref{lem:eps1p7}, \ref{lem:eps2p7} and \ref{lem:eps2p11} and that $G'$ is connected from Lemma \ref{lem:connected}. 
Let $Q$ as before, that is, $x_Q$, the number of $2$-secants to $\S$ through $Q$ is minimal. We take a scalar multiple of $\c$ such that $Q\in K_{p-1}$. Since this implies that $1$ is a vertex of $G'$ too, and $|\K|=2r\neq p-1$, we know that the colours used are $1,p-1,2,p-2,\ldots,r,p-r$ with $2\leq r< \frac{p-1}{2}$. Let $H$ be the set of lines through $Q$ that are not a $2$-secant to $\S$ and which contain at most one point, different from $Q$, not contained in $K_{p-r}$. Let $|H|=x$. Then every line through $Q$, not in $H$ and not a $2$-secant, contains at least $2$ points not in $K_{p-r}$. There are $p^2+1-x-x_Q$ such lines, each containing at least $2$ points not belonging to $K_{p-r}$, different from $Q$. Furthermore, the point $Q$ does not belong to $K_{p-r}$, nor do the $x_Q$ points of $K_1$ that lie on $2$-secants through $Q$, thus giving the following lower bound for the points of $\S$ not in $K_{p-r}$:
\[
	2(p^2+1-x-x_Q)+x_Q+1\leq |\S|-|K_{p-r}|\:.
\]
Note that we used that $p-r\neq 1$ and $r\neq 1$.
Rewriting this expression, using that  $|\S|\leq 2p^2-2p+4$ and $|K_{p-r}|\geq x_Q$, yields that $x\geq p$. Let $\ell$ be a line of $H$. Suppose that the point on $\ell$, different from $Q$ and not necessarily contained in $K_{p-r}$ is contained in $K_s$. If $\ell$ has $y$ points of $K_{p-r}$, then, since $(\c,\ell)=0$, we have that $p-1+y(p-r)+s\equiv 0\pmod{p}$, so $s\equiv 1+yr\pmod{p}$. Either $1+yr>p$, and then $y>\frac{p-1}{r}$, or $0< 1+yr\leq p$ and $s=1+yr$. But since $1+yr>r$ and there are no points with colour in $\{r+1,r+2,\ldots,p-r-1\}$, we have that $1+yr\geq p-r$. The latter implies that $y\geq \frac{p-1}{r}-1$. We conclude that every line of $H$ contains at least $\frac{p-1}{r}-1+2=\frac{p-1}{r}+1$ points of $\S$. In particular, every line of $H$ contains more than $3$ points of $\S$. So from Equation \eqref{eq3}, using that there are at least $p$ lines in $H$, we obtain that $x_Q\geq 2p-1+p(\frac{p-1}{r}-2)=\frac{p(p-1)}{r}-1$. This leads to
\begin{align}
	p^2-p-r\leq rx_Q=\frac{|\K|}{2}x_{Q}\leq\frac{1}{2}\sum_{i\in\K}|K_{i}|=\frac{|\S|}{2}\leq p^2-p+2.\label{rxq}
\end{align}
First suppose that $r\nmid p-1$. Since every line of $H$ contains at least $\frac{p-1}{r}$ points of $\S$, different from $Q$, and $r \nmid p-1$ and $r\nmid p$, we find that every line of $H$ contains at least $\frac{p+1}{r}$ points of $\S$, different from $Q$. Counting the points of $\S$ per line through $Q$ shows that
\begin{align*}
	|\S| &\geq 1+x_Q+x\left(\frac{p+1}{r}\right)+2(p^2+1-x_Q-x)=2p^2+3+\frac{x(p-2r+1)}{r}-x_Q\\
	&\geq 2p^2-2p+3+\frac{p(p+1)}{r}-x_Q,
\end{align*}
since $x\geq p$. Since $|\S|=2p^2-2p+2+\epsilon$ it follows that $\frac{p(p+1)}{r}-x_Q\leq\epsilon-1$, a contradiction since $x_Q\leq \frac{p^2-p+2}{r}$.

So $r|p-1$, which implies that $r\leq \frac{p-1}{3}$; indeed recall that $r<\frac{p-1}{2}$. From Equation \eqref{rxq}, we get that $x_Q=\frac{p(p-1)}{r}-1$, or $x_Q=\frac{p(p-1)}{r}$, or $r=2$ and $x_Q=\frac{p(p-1)}{2}+1$.

If $x_Q=\frac{p(p-1)}{r}-1$, then counting the points in $\S$ on the lines through $Q$ yields that
\begin{align*}
	|\S| &\geq 1+x_Q+x\left(\frac{p-1}{r}\right)+2(p^2+1-x_Q-x)=2p^2+3+\frac{x(p-2r-1)}{r}-\left(\frac{p(p-1)}{r}-1\right)\\&\geq 2p^2-2p+4,
\end{align*}
since $x\geq p$. So if $\epsilon=1$, this immediately is a contradiction, and if $\epsilon=2$, this bound is sharp. So we find that, for $\epsilon=2$, there are precisely $p$ lines through $Q$ all of whose $\frac{p-1}{r}$ points, except for $Q$, belong to $K_{p-r}$, and that all other lines through $Q$ are $2$-or $3$-secants. By considering these 2-secants and 3-secants, we see that $|K_{p-1}|=|K_2|+1$, that $|K_1|+1=\frac{p(p-1)}{r}=|K_{p-r}|$ and that all colours $\mu\in\K\setminus\{1,2,p-1,p-r\}$ have $|K_\mu|=|K_{p-\mu+1}|$. 

 We note that $|K_{1}|=x_{Q}$, so by Lemma \ref{KvsX}(i) and the minimality of $x_{Q}$, we have that $x_{A}=x_{Q}=|K_{1}|$ for any $A\in K_{p-1}$. So, we can repeat the previous argument for any $A\in K_{p-1}$. In particular, we see that any line containing a point of $K_{p-r}$ and a point of $K_{p-1}$ contains precisely $\frac{p-1}{r}$ points of $K_{p-r}$ and no further points. Since $|K_2|\geq x_Q$ and $|K_{p-1}|=|K_2|+1$, we have that $|K_{p-1}|\geq \frac{p(p-1)}{r}$. Now consider a point $R\in K_{p-r}$, then we know that each of the $|K_{p-1}|$ points of $K_{p-1}$ determines a different line through $R$, each containing $\frac{p-1}{r}-1$ points of $K_{p-r}$, different from $R$. It follows that the number of points in $K_{p-r}$ is at least 
\begin{align*}
\frac{p(p-1)}{r}\left(\frac{p-1}{r}-1\right)+1
\end{align*}
But $|K_{p-r}|=|K_1|+1=\frac{p(p-1)}{r}$, a contradiction since $r\leq \frac{p-1}{2}$.

Now assume that $x_Q=\frac{p(p-1)}{r}$. By Lemma \ref{KvsX}(i) and the assumption on $Q$, we know that each of the $2r$ colours occurs at least $x_Q=\frac{p(p-1)}{r}$ times. Now consider a point $R\in K_{p-r}$. Suppose that $x_R\geq x_Q+5$, then $|K_r|\geq x_Q+5$ and we find that the number of points in $\S$ is at least $(2r-1)\frac{p(p-1)}{r}+\frac{p(p-1)}{r}+5=2p^{2}-2p+5$, a contradiction. Hence, $x_R-x_Q\leq 4$.

A point $T\in K_{p-1}$ does not lie on a $3$-secant to $\S$ through $R$. Furthermore, no $2$-secant through $R$ contains a point of $K_{p-1}$. Let $H'$ be the set of lines through $R$ that are not $2$-secants to $\S$ and contain, apart from $R$, at most one point not in $K_{p-1}$, contained in $K_s$.  Let $x'=|H'|$. We again find that
$$|\S|-|K_{p-1}|\geq 1+x_R+2(p^2+1-x_R-x').$$
Since $|K_{p-1}|\geq x_Q$, it follows that $x'\geq \frac{x_Q-x_R+2p-1}{2}\geq\frac{2p-5}{2}>0$.

Let $\ell$ be a line of $H'$. If $\ell\setminus\{R\}$ contains only points of $K_{p-1}$, then there are at least $p-r$ points in $\ell\setminus\{R\}$. If there is a point $N$ in $\ell\setminus\{R\}$ contained in $K_s$, $s\neq p-1$ then $p-r+y'(p-1)+s\equiv 0\pmod{p}$, where $y'$ is the number of points of $K_{p-1}$ on $\ell$. It follows that $s\equiv y'+r\pmod{p}$. Since $y'+r>r$, we have that $y'+r\geq p-r$, so $y'\geq p-2r$. Since $p-r\geq p-2r+1$, we conclude that every line of $H'$ contains at least $p-2r+1$ points of $\S$, different from $R$.

Counting the points of $\S$ per line through $R$ yields:
\begin{align*}
	|\S|&\geq 1+x_R+x'\cdot(p-2r+1)+2(p^2+1-x_R-x')=2p^{2}+3-x_R+x'\cdot(p-2r-1)\\
	&\geq 2p^{2}+3-x_R+\frac{x_Q-x_R+2p-1}{2}(p-2r-1)\\
	&=2p^{2}-2p+4-x_{Q}+\frac{x_Q-x_R+2p-1}{2}(p-2r+1)\:.
\end{align*}
Using $|\S|\leq 2p^2-2p+4$ leads to 
\begin{align}2x_Q\geq (p-2r+1)(x_Q-x_R+2p-1).\label{countend}\end{align}
Recall that $x_R-x_Q\leq 4$, so $x_Q-x_R+2p-1\geq 2p-5$. As $r\leq \frac{p-1}{3}$ also $p-2r+1\geq \frac{p+5}{3}$.

Thus, $2x_Q\geq\frac{p+5}{3}(2p-5)$. But $2x_Q=  \frac{2p(p-1)}{r}\leq \frac{p(p-1)}{2}$ if $r\geq 4$, which leads to a contradiction since $p\geq 5$. Furthermore, if $r=3$, then Equation \eqref{countend} gives that
$\frac{2p(p-1)}{3}\geq (p-5)(2p-5)$, a contradiction if $p\geq11$, and if $r=2$, we find that
$\frac{2p(p-1)}{2}\geq (p-3)(2p-5)$, a contradiction if $p\geq 11$.  Note that for $p=7$, the case $r=3$ does not occur since we assumed that $r<\frac{p-1}{2}$.  
We now deal with the case $r=2$ and $p=7$. We have $x_Q=21$ and $R\in K_{5}$. Note that every $3$-secant through $R$ necessarily contains two points of $K_1$  (since the colours are $1,6,2,5$), so $y_{R}\leq\frac{|K_{1}|}{2}$. From Equation \eqref{eq1}, it follows that
\begin{align*}
	63&\leq y_{R}+2x_{R}\leq \frac{|K_{1}|}{2}+2x_{R}= \frac{|\S|-|K_{2}|-|K_{5}|-|K_{6}|}{2}+2x_{R}\\
	&\leq \frac{88-x_{R}-2x_{Q}}{2}+2x_{R}=44-x_{Q}+\frac{3x_{R}}{2}\\
	&\leq 44+\frac{3}{2}\cdot 4+\frac{x_{Q}}{2}=60+\frac{1}{2},
\end{align*}
a contradiction.

Finally, assume that $r=2$ and $x_Q=\frac{p(p-1)}{r}+1=\frac{p(p-1)}{r}+1$. By Lemma \ref{KvsX}(i) and the assumption on $Q$, we know that each of the $4$ colours $1,p-1,2,p-2$ occurs at least $x_Q=\frac{p(p-1)}{2}+1$ times. Since $|\S|=4x_{Q}$, we find a contradiction if $\epsilon=1$ and that each colour occurs exactly $\frac{p(p-1)}{2}+1$ times if $\epsilon=2$. We immediately also have that $x_{A}=\frac{p(p-1)}{2}+1$ for any point $A\in\S$. Now consider a point $R\in K_{p-2}$.

A point in $K_{p-1}$ does not lie on a 2-secant or 3-secant to $\S$ through $R$, since $r+1$ is not a colour. Let $H''$ be the set of lines through $R$ that are not $2$-secants to $\S$ and contain, apart from $R$, at most one point not in $K_{p-1}$.  Let $x''=|H''|$. We again find that
\[
	|\S|-|K_{p-1}|\geq 1+x_R+2(p^2+1-x_R-x'').
\]
Since $|K_{p-1}|=x_R$, it follows that $x''\geq \frac{2p-1}{2}$, hence $x''\geq p>0$. 
Let $\ell$ be a line of $H'$. If $\ell\setminus\{R\}$ contains only points of $K_{p-1}$, then there are at least $p-2$ points in $\ell\setminus\{R\}$. 
If there is a point $N$ in $\ell\setminus\{R\}$ contained in $K_s$, $s\neq p-1$, then either $N\in K_2$ and there are at least $p$ points of $K_{p-1}$ in $\ell\setminus\{R\}$, or $N\in K_{p-2}$ and there are at least $p-4$ points of $K_{p-1}$ in $\ell\setminus\{R\}$. We conclude that every line of $H''$ contains at least $p-3$ points of $\S$, different from $R$.

Counting the points of $\S$ per line through $R$ yields:
\begin{align*}
	|\S|&\geq 1+x_R+x''\cdot(p-3)+2(p^2+1-x_R-x'')=2p^{2}+3-x_R+x''\cdot(p-5)\\
	&\geq 2p^{2}+3-x_R+p(p-5)
\end{align*}
Using $|\S|=2p^2-2p+4$ leads to 
\begin{align*}
	x_R\geq p(p-5)+2p-1=p^{2}-3p-1,
\end{align*}
a contradiction since $x_{R}=\frac{p(p-1)}{2}+1$ and $p\geq7$.
\end{proof}

The Main Theorem (see page \pageref{mt:maintheorem}) now directly follows from Corollary \ref{2coloursnieuw} and Theorem \ref{final}. Corollary \ref{cor:main} then follows from Corollary \ref{cordesmin}, the Main Theorem and Theorem \ref{noembeddedantipodal}.


\begin{thebibliography}{99}

\bibitem{AK} E.F. Assmus, Jr. and J.D. Key. {\em Designs and their Codes}. Cambridge University Press, 1992. 
\bibitem{bagchi} B. Bagchi and S. P. Inamdar. Projective geometric codes. {\em J. Combin. Theory Ser. A} {\bf 99(1)} (2002), 128--142.
\bibitem{p5} K.L. Clark, L.D. Hatfield, J.D. Key and H.N. Ward. Dual codes of projective planes of order 25. {\em Adv Geom} {\bf 3} (2003), 140--152.
\bibitem{dembowski} P. Dembowski. {\em Finite Geometries}. Springer-Verlag, Berlin-New York, 1968.
\bibitem{p3} J.D. Key and M.J de Resmini. Ternary dual codes of the planes of order nine. {\em J. Statist. Plan Inference} {\bf 95} (2001), 229--236.
\bibitem{keymac} J.D. Key and K. Mackenzie. An upper bound for the $p$-rank of a translation plane. {\em J. Combin. Theory, Ser. A} {\bf 56(2)} (1991), 297--302. 
\bibitem{p7} J.D Key and F. Ngwane. A lower bound for the minimum weight of the dual $7$-ary code of a projective plane of order $49$. {\em Des. Codes Cryptogr.} {\bf 44} (2007), 133--142.
\bibitem{payne} S. Payne. On the non-existence of a class of configurations which are nearly generalized $n$-gons. {\em J. Combin. Theory, Ser. A} {\bf 12(2)} (1972), 268--282.
\bibitem{sachar}  H. Sachar. The $F_p$ span of the incidence matrix of a finite projective plane. {\em Geom. Dedicata} {\bf 8(4)} (1979), 407--415.
\bibitem{schleier} A. Schleiermacher. On a class of partial planes related to biplanes. {\em J. Geom.} {\bf 107(2)} (2016), 445--466.
\end{thebibliography}
\end{document}